\documentclass[11pt,twoside]{article}
\usepackage{amsmath,amssymb,amsthm,mathptmx,hyperref,color}
\usepackage[english]{babel}%
\usepackage[inner=3.4cm,outer=3.65cm,bottom=3cm]{geometry}
\usepackage{xcolor}
\usepackage{mathtools,amsfonts,amsthm}
\usepackage{enumitem}

\usepackage{color}

\theoremstyle{remark}
\newtheorem{remark}{Remark}[section]
\theoremstyle{definition}
\newtheorem{theorem}{Theorem}[section]
\newtheorem{definition}[theorem]{Definition}

\newtheorem{proposition}[theorem]{Proposition}
\newtheorem{lemma}[theorem]{Lemma}
\newtheorem{corollary}[theorem]{Corollary}
\newtheorem{hypothesis}[theorem]{Hypothesis}

\DeclareMathOperator{\R}{\mathbb{R}}
\DeclareMathOperator{\Rr}{\mathbb{R}}

\DeclareMathOperator{\W}{\mathbb{W}}
\DeclareMathOperator{\E}{\mathcal{E}}
\DeclareMathOperator{\Y}{\mathbb{Y}}
\DeclareMathOperator{\MO}{\mathbb{O}}
\DeclareMathOperator{\MI}{\mathbb{I}}
\DeclareMathOperator{\C}{\mathcal{C}}

\DeclareMathOperator{\N}{\mathbb{N}}

\DeclareMathOperator{\dom}{dom}
\DeclareMathOperator{\loc}{loc}

\DeclareMathOperator{\ra}{\rightarrow}
\newcommand{\de}{\text{d}}
\DeclareMathOperator{\tr}{tr}

\newcommand{\sym}{\text{sym}}
\newcommand{\skw}{\text{skw}}

\newcommand{\f}[1]{{\pmb{ #1}}}
\DeclareMathOperator{\di}{\nabla\cdot}

\newcommand{\ov}[1]{\overline{{#1}}}
\newcommand{\un}[1]{\underline{{#1}}}

\renewcommand{\t}{\partial_t}

\DeclareMathOperator*{\argmin}{arg\,min}

\newcommand{\vv}{\f \varphi}

\newcommand{\bft}{\f{\sigma}}

\newcommand{\tv}{\tilde{\f v}}
\newcommand{\tbF}{\tilde{\bF}}
\newcommand{\tbB}{\tilde{\bB}}

\DeclareMathOperator{\V}{\mathbb V}

\DeclareMathOperator{\Hsig}{\f H^{1}_{0,\sigma}}

\DeclareMathOperator{\Ha}{\f L^2_{\sigma}}

\newcommand{\diver}{\ensuremath{\operatorname{div}}}
\DeclareMathOperator{\inter}{int}
\DeclareMathOperator{\BVl}{{BV}([0,\infty))}
\newcommand{\eT}{\infty}
\newcommand{\otime}{(0,\infty)}
\newcommand{\cltime}{[0,\infty)}
\newcommand{\otimeT}{(0,T)}
\newcommand{\cltimeT}{[0,T]}

\newcommand{\dreidots}{\text{\,\multiput(0,-2)(0,2){3}{$\cdot$}}\,\,\,\,}

\newcommand{\dreidotkom}{\text{\,\multiput(0,0)(0,2){2}{$\cdot$}\put(0,0){,}}\,\,\,\,}

\usepackage{stackengine}

\usepackage[normalem]{ulem}

\newcommand{\tU}{\tilde{\f U}}

\renewcommand{\di}{\nabla \cdot }
\newcommand{\e}{\text{\textit{e}}}
\newcommand{\bB}{\mathbb{B}}
\newcommand{\bF}{\mathbb{B}}
\newcommand{\diss}{\Psi}

\newcommand{\Banach}{\V}
\newcommand{\Banachdual}{{\V^*}}

\author{Abramo Agosti, Robert Lasarzik, and Elisabetta Rocca}

\date{\today}
\title%
{Energy-variational solutions for viscoelastic fluid models}	

\begin{document}
\maketitle
\begin{abstract}
In this article, we introduce the concept of \textit{energy-variational solutions} for a large class of systems of nonlinear evolutionary partial differential equations.
Under certain convexity assumptions, the existence of such solutions can be shown constructively by an adapted minimizing movement scheme. Weak-strong uniqueness follows by a suitable relative energy inequality. 


Our main motivation is to apply the general framework to viscoelastic fluid models. Therefore, we give a short overview on different versions of such models and their derivation. The abstract result is applied to two of these viscoelastic fluid models in full detail. In the conclusion, we comment on further applications of the general theory and its possible impact. 

\end{abstract}
\noindent

\noindent
\begin{tabular}{ll}\textbf{MSC2020:} &
35A15, 35D99, 35M33, 35Q35,   76A10.
\\
\textbf{Keywords:} &
Energy-variational solutions, existence, minimizing movements, weak-strong\\& uniqueness, viscoelastic fluids, Oldroyd-B.
\end{tabular}

\section{Introduction}
Viscoelastic fluids can be seen as an in-between-state of Newtonian fluids and elastic solids. In elastic materials, the resistance to deformation is described in terms of  the deformation and in viscous fluids in terms of displacement rates. A combination of both leads to characteristics like a nonlinear stress strain relation or even hysteresis in the stress-strain curve, creep,  or memory effects in the material (see for instance~\cite{DynPolyLiq,hu,lin}). 

A standard model for viscoelastic fluids, is the so-called Oldroyd-B model, which couples the Navier-Stokes equations for the fluid velocity $\f v:\ov\Omega \times [0,\infty )\ra \R^d $  with an additional evolution equation for the left Cauchy--Green tensor~$\bB : \ov\Omega \times [0,\infty )\ra \R^{d\times d}$,
\begin{subequations}\label{Oldroyd}
\begin{align}
\t \f v + ( \f v \cdot\nabla) \f v + \nabla p - \mu \Delta \f v + \beta  \di \bB &= \f f \,, \quad &\di \f v &= 0 \,,\label{Oldroy1}\\
\t \bB + ( \f v \cdot \nabla )\bB  - 2((\nabla \f v)_{\skw} \bB)_{\sym} -\alpha (( \nabla \f v)_{\sym} \bB )_{\sym} + \bB - \MI &= 0 \,, \quad &\bB &= \bB^T\,, \label{Oldroyd2}
\end{align}
endowed with appropriate initial and boundary conditions, where $p$ denotes the pressure of the system, $\mu \in (0,\infty)$ and $\alpha\,,\beta \in\R$ having the same sign. Finally, $\f f$ represents a given source.

If $\alpha\neq 0$, this model formally satisfies the associated dissipation mechanism 
\begin{align}
\mathcal{E}_O( \f v , \bB) \Big|_s^t+\int_s^t \Psi_O(\f v , \bB)
\de \tau = \int_s^t\langle \f f , \f v \rangle \de \tau  \,,\label{OldroydEn}
\end{align}
\end{subequations}
where the corresponding energy is given by 
$$ \mathcal{E}_O( \f v , \bB) = \int_\Omega\frac{1}{2}| \f v|^2 +\frac{\beta}{\alpha} \tr (\bB - \MI - \ln \bB)\de x  $$ 
and the dissipation functional by 
$$ \Psi_O(\f v , \bB) : =\int_\Omega2\mu \lvert (\nabla \f v)_{\sym}\rvert^2 +\frac{\beta}{\alpha}  \tr(\bB ( \MI - \bB^{-1})^2)\de x .$$ 

From a mathematical point of view, analysis for these kinds of models is rather challenging. First and foremost  the quadratic terms in~\eqref{Oldroyd2} are not known to be integrable from the \textit{a priori} bounds of~\eqref{OldroydEn}, since these only assure an $L^1$-bound for the stress tensor $\bB$. Secondly, the highly nonlinear terms are not treatable by compactness methods, since no compactness due to embeddings is known to hold for $\bB$. 
For $\alpha=0$, Lions and Masmoudi~\cite{Lions} showed existence of global weak solutions to \eqref{Oldroy1}--\eqref{Oldroyd2} by a technique called propagation of compactness (see also~\cite{fenep}). 
Many other works in the field introduce so-called stress diffusion $-\Delta \bB$ on the left hand-side  of~\eqref{Oldroyd2} in order to improve 
compactness~\cite{mielke,mizerova}. 
This stress diffusion can be rigorously introduced by thermodynamical modeling~\cite{stressdiffusion}, and can serve to model the dispersion and attenuation of elastic waves which may be experimentally observed in some particular situations \cite{roubicek2}, but often it is  added in order to infer strong convergence properties for a suitable approximating sequence of $\bB$.
Moreover, there are numerical studies showing that numerical simulations for viscoelastic models are close to empirical data in the absence of stress diffusion effects~\cite{moresi} or the simulated effects are closer to the expected phenomena for vanishing stress diffusion~\cite{agosti2023}.

In the current article, we therefore propose a novel analytical framework to treat especially viscoelastic models without stress diffusion. We introduce the concept of \textit{energy-variational solutions} for a general class of evolutionary PDEs of the form
\begin{align}
\t \f U + A(\f U) = \f 0 \quad \text{ in }\Y^* \text{ with }\quad\f U(0)=\f U_0 \quad \text{in }\V\,\label{introGEnEq}
\end{align} 
for two Banach spaces $\V$ and $\Y$ such that $ \Y \subset \V^*$, $A\,:\,[0,\infty)\times \V\to \Y^*$, $\f U_0\in \V$. 
A smooth solution to~\eqref{introGEnEq} for any initial value formally fulfills the energy dissipation mechanism 
\begin{align}
\E( \f U) \Big|_s^t +\int_s^t \diss(\f U) \de \tau \leq  0 \label{intreoEnGen}
\end{align}
  for all $ s<t\in\cltime$, where $\E$ and $\diss$ are suitable energy and dissipation functionals defined on $\V$. 
A standard weak solution to~\eqref{introGEnEq} fulfills
\begin{align}
\langle \f U , \Phi \rangle \Big|_s^t -\int_{{s}}^t \langle \f U , \t \Phi \rangle + \langle A(\f U) , \Phi \rangle \de \tau = 0 \,\label{introweak}
\end{align}  
  for all $ s<t\in\cltime$ and $\Phi \in \C^1(\cltime;\Y)$. Introducing an upper bound for the energy $ E \geq \E( \f U)$, subtracting~\eqref{introweak} from~\eqref{intreoEnGen}, and, adding a term for the associated energy defect with weight $\mathcal{K}\,:\, \Y\to [0,+\infty)$, gives the formulation
  \begin{equation}  \label{introEnVar}
          \left[ E - \langle \f U , \Phi\rangle \right]\Big|_s^t + \int_s^t \langle \f U, \t \Phi \rangle 
  +
   \diss(\f U) - \langle A(\f U) , \Phi\rangle + \mathcal{K}(\Phi) \left [\E(\f U) -E\right ]\de \tau \leq 0 
  \end{equation}
  for a.e. $ {s}<t\in\cltime$ and all $\Phi \in \C^1(\cltime;\Y)$.  
~\eqref{introEnVar} is a convex function in $\f U$. Besides this existence of solutions, which is proven constructively via a time-discretization based on sequential minimization, we also provide the weak-strong uniqueness of solutions. This means that all \textit{energy-variational solutions} coincide with a local strong solution emanating from the same initial datum, as long as the latter exists.  Moreover, the solution is a semi-flow. Up to our knowledge, this result is the first one entailing global existence and weak-strong uniqueness result for such a general class of models. 

In comparison to more established solutions concepts, \textit{energy-variational solutions} can be seen as an in-between state of dissipative solutions and measure-valued solutions. In dissipative solutions, the formulation is generalized to an inequality~\cite[Sec.~4.4]{lionsfluid} while  in measure-valued solutions an auxiliary measure is introduced such that the equations is fulfilled in the limit~\cite{DiPernaMajda}. \textit{Energy-variational solutions} combine both approaches by adding an auxiliary variable  relaxing the formulation to an inequality. Nevertheless, the concept has many desirable properties, for instance: the existence and weak-strong uniqueness of \textit{energy-variational solutions} for a large class of models via a constructive existence proof. For certain systems in fluid dynamics, the equivalence of measure-valued solutions and \textit{energy-variational solutions} is known~\cite{hyper}, even though the degrees of freedom are heavily reduced in the second one, since the auxiliary matrix-valued measure is replaced by a real number in every point in time. Moreover, in~\cite{max} it was observed that \textit{energy-variational solutions} are very flexible, in the sense that they allow to identify the limit of numerical schemes, while this result  seems to be out of reach with measure-valued solutions. 
  Moreover, the set of  \textit{energy-variational solutions}  is weakly-star closed and in some cases set-valued continuous, which enables to define minimization schemes on the solution set as proposed in~\cite{envar}.

  Our main result, similar to the one of \cite{hyper} related to hyperbolic conservation laws, introduces a stronger solution concept with respect to the one of dissipative solutions~\cite{lionsfluid,diss} and thus also stronger than the solution concept proposed in~\cite{katharina}. Moreover, the proposed result has potentially multiple applications outside of viscoelastic fluid dynamics, like, for instance, the ones  considered in~\cite{hyper} about the compressible and incpompressible  Euler equations as well as the incompressible magnetohydrodynamics.
  
After a brief introduction to different variants of viscoelastic models, the general result is applied to two viscoelastic fluid models in full detail. 
The first model we consider is taken from~\cite{Malek}, but allowing for vanishing stress diffusion. As in \cite{Malek}, in case where  the  energy in~\eqref{OldroydEn}  is regularized by a quadratic term, leading to the system
 \begin{subequations}
 \label{sysviscointro}     
 \begin{align*}
\t \f v + ( \f v \cdot \nabla ) \f v + \nabla p - \mu \Delta \f v  - \alpha  \di \left ( (1-\beta) (\bF-\mathbb I)  +\beta ( \bF^2 - \bF) \right ) ={}& \f f \,, \quad \di \f v = 0 \,,
\\
\t \bF +  ( \f v \cdot \nabla ) \bF -2[ (\nabla \f v)_{\skw} \bF]_{\sym} - \alpha   [  ( \nabla \f v )_{\sym} \bF]_{\sym}  + (\mathbb I+\delta \bF) (\bF - \mathbb I) 
={}& 0 \,, \quad  ( \bF)_{\skw} = 0 \,, 
\end{align*}
with $\mu$, $\delta \in (0,\infty)$,  $\alpha \in \R$, $\beta \in (0,1)$, and the associated energy 
\begin{align*}
\E_Q(\f v ,\bB) :=  \int_\Omega\frac{1}{2}|\f v|^2 + (1-\beta) \tr( \bB - \MI - \ln(\bB))  + \frac{\beta}{2}| \bB-\MI|^2 \de x \,
\end{align*}
 \end{subequations}
and dissipation 
$$
\Psi_Q (\f v, \bB) := \int_\Omega \mu | \nabla \f v|^2 + (1-\beta) \bB:(\MI-\bB^{-1})^2 + (\beta + \delta(1-\beta)) | \bB-\MI|^2 + \delta \beta \bB : (\bB -\MI)^2\de x \,,
$$
such that an associated energy-dissipation relation like in~\eqref{OldroydEn} is formally fulfilled for solutions to~\eqref{sysviscointro} as well. 
We prove existence and weak-strong uniqueness of \textit{energy-variational solutions} to this system by applying the aforementioned general result.

Secondly, we consider a model inspired by~\cite{perrotti} for a symmetrized Neo--Hookean approach, which leads to the equations for the deformation gradient $\mathbb F$, the velocity field  $\f v$, and pressure $p$,
\begin{subequations}\label{intro}
 \begin{align}
\t \f v + ( \f v \cdot \nabla ) \f v + \nabla p - \mu \Delta \f v  - \alpha  \di \left ( \mathbb F^2- \mathbb I  \right ) ={}& \f f \,, \quad \di \f v = 0 \,,\label{intro1}\\
\t \mathbb F +  ( \f v \cdot \nabla ) \mathbb F -[ (\nabla \f v)_{\skw} \mathbb F]_{\sym} - \alpha [  ( \nabla \f v )_{\sym} \mathbb F]_{\sym}  +  ( \mathbb F - \mathbb F^{-1}) ={}& 0 \,, \quad  ( \mathbb F)_{\skw} = 0 \, .\label{intro2}
\end{align}
with $\mu \in (0,\infty)$, $\alpha \in \R$, and the associated energy 
\begin{align*}
   \E_{S}(\f v ,\mathbb F) :=  \frac{1}{2}\int_\Omega \lvert \f v\vert ^2 + \lvert\mathbb F \rvert^2 - | \mathbb I |^2 - \ln \det (\mathbb F^2) \de x \,
\end{align*}
\end{subequations}
{and dissipation }
$$
\Psi_{S}(\f v, \mathbb F ) := \int_\Omega \mu | \nabla \f v |^2 + | \mathbb F - \mathbb F^{-1}|^2 \de x = \int_\Omega \mu | \nabla \f v |^2 +\tr(\mathbb F^2 - 2\MI + \mathbb F^{-2}) \de x,
$$
such that an associated energy-dissipation relation like in~\eqref{OldroydEn} is formally fulfilled for solutions to~\eqref{sysviscointro} as well. 
Also for this model we prove existence and weak-strong uniqueness of \textit{energy-variational solutions}, based on our general result in Banach spaces (cf.~Theorem~\ref{thm:envar}).

In the conclusions, we mention different other models that fit into the proposed abstract framework and give an outlook on possible future research directions.

We note  that we prove the abstract existence result by a constructive time-discretization approach. Similar to the minimizing movements scheme for Gradient flows, we define a time incremental minimization algorithm for an approximation of the solution. We think that this is not the only similarity to the gradient flow setting, where a lot of results were achieved in recent years, like general existence and uniqueness results~\cite{gradientflow} and singular limit results bridging between different scales~\cite{scales}. We think that the general approach presented in this article has the potential to generalize such results toa more general framework.

\section{Preliminaries}
\paragraph*{Notation}
We denote by $\mathcal{C}_{c,\sigma}^\infty(\Omega;\Rr^3)$ the space of smooth solenoidal functions with compact support. By $ L^p_{\sigma}( \Omega) $, $\Hsig(\Omega)$,  and $  W^{1,p}_{0,\sigma}( \Omega)$, we denote the closure of $\mathcal{C}_{c,\sigma}^\infty(\Omega;\Rr^3)$ with respect to the norm of $ L^p(\Omega) $, $  H^1( \Omega) $, and $  W^{1,p}(\Omega)$, respectively. 
By $  \mathbb{R}^{d\times d}$ we denote $d$-dimensional quadratic matrices, by $  \mathbb{R}^{d\times d}_{\sym}$ and $  \mathbb{R}^{d\times d}_{\skw}$ the symmetric and skew-symmetric subsets, and by $  \mathbb{R}^{d\times d}_{\sym,+}$ the symmetric positive definite matrices. 
The identity matrix is denoted by $\mathbb{I} $ and the zero matrix by $\mathbb{O}$.
The symmetric and skew-symmetric part of a matrix $\mathbb A\in \Rr^{d\times d}$ are denoted by $(\mathbb A)_{\sym} $ and $(\mathbb A)_{\skw}$, respectively. 
We denote $L^p_{\sym}(\Omega) :=  L^p(\Omega;\R^{d\times d}_{\sym})$ and $L^p_{\sym,+}(\Omega) :=  L^p(\Omega;\R^{d\times d}_{\sym,+})$ for $p\in[1,\infty]$.
For a given Banach space $\mathbb{X}$, we denote by $\mathbb{X}^*$ its dual space, the space $\C_w([0,T];\mathbb X )$ denotes the functions on $[0,T]$ taking values in $\mathbb X$ that are continuous with respect to the weak topology of $\mathbb X$. For two Banach spaces $\mathbb X$, $\mathbb Y$, we denote the set of linear continuous mappings from $\mathbb X$ with values in $\mathbb Y$ by $\mathcal{L}(\mathbb X, \mathbb Y)$. 
The total variation of a function $E:[0,\infty)\ra \Rr$ is given by 
$$ \vert E \vert_{\text{TV}([0,\infty))}= \sup_{0=t_0<\ldots <t_n<\infty} \sum_{k=1}^n \lvert E(t_{k-1})-E(t_k) \rvert\,, $$
where the supremum is taken over all finite partitions of the interval $[0,\infty)$. 
We denote the space of all bounded functions of bounded total variations on $[0,\infty)$ by~$\BVl$. 
Note that the total variation of a monotone decreasing nonnegative function $E$ only depends on the initial value, \textit{i.e.,}
\begin{align*}
\vert E\vert_{\text{TV}([0,\infty))} = \sup_{0=t_0<\ldots <t_n<\infty}\sum_{k=1}^n\lvert E(t_{k-1})-E(t_k)\rvert \leq E(0) - E(t_n) \leq E(0) \,.
\end{align*}

Let $\V$ be a Banach space and $\E :\V  \to [0,\infty]$ be a convex, lower semi-continuous function. 
The domain of $\E$ is defined by $\dom\E=\{\f v\in\V\mid\E(\f v)<\infty\}$.
We denote the convex conjugate of $\E$ by $\E^\ast$, 
which is defined by 
\begin{align*}
\E^*(\f z) =\sup_{\f y \in \Banach} \left [
\langle\f z,\f y\rangle
- \E(\f y)\right ] \qquad \text{for all }\f z \in \Banachdual \,.
\end{align*}
Then $\E^*$ is also convex and lower semi-continuous.
We introduce the subdifferential $\partial \E$ of $\E$ by 
\begin{align*}
\partial \E (\f y) := \left  \{ \f z \in \Banachdual \mid 
\forall \tilde{\f y} \in \Banach :\ \E( \tilde{\f y}) \geq \E (\f y) +
\langle \f z , \tilde{\f y} - \f y\rangle
\right \} \,
\end{align*}
for $\f y\in\Banach$.
The subdifferential $\partial\E^\ast$ of $\E^\ast$ is defined analogously.
Then the Fenchel equivalences hold:
For $\f y \in \Banach$, $\f z \in \Banachdual$ we have
\begin{equation}\label{eq:fenchel}
\f z \in \partial \E(\f y)
\quad\iff\quad 
\f y \in \partial \E^*(\f z )
\quad\iff\quad 
\E(\f y ) + \E^*(\f z) = 
\langle\f z , \f y\rangle\,.
\end{equation}
A proof of this well-known result can be found in~\cite[Prop~2.33]{barbu} for example.
If $\partial\E(\f y)$ is a singleton for some $\f y\in \V$, 
then $\E$ is Fr\'echet differentiable in $\f y$ and 
$\partial\E(\f y)=\{D\E(\f y)\}$.
In this case, we identify $\partial\E(\f y)$ with $D\E(\f y)$.



\begin{lemma}\label{lem:invar}
Let $f\in L^1_{\loc}(0,\infty)$, $g\in L^\infty_{\loc}(0,\infty) $ and $g_0\in\R$.
Then the following two statements are equivalent:
\begin{enumerate}[label=\roman*.]
\item
The inequality 
\begin{equation}
-\int_0^\infty\phi'(\tau) g(\tau) \de \tau  + \int_0^\infty \phi(\tau) f(\tau) \de \tau - \phi(0)g_0 \leq 0 
\label{ineq1}
\end{equation}
holds for all $\phi \in {\C}^1_c (\cltime)$ with $\phi \geq 0$.
\item
The inequality
\begin{equation}
    g(t) -g(s) + \int_s^t f(\tau) \de \tau \leq 0 
    \label{ineq2}
\end{equation}
holds for a.e.~$s< t\in\cltime$,
including $s=0$ if we replace $g(0)$ with $g_0$.
\end{enumerate}
If one of these conditions is satisfied, 
then $g$ can be identified with a function in $\BVl$  such that
\begin{equation}
    g(t+) -g(s-) + \int_s^t f(\tau) \de \tau \leq 0 \,
    \label{ineq.pw}
\end{equation}
for all $s\leq t\in[0,\infty)$,
where we set $g(0-)\coloneqq g_0$.
In particular, it holds $g(0+)\leq g_0$ and $g(t+)\leq g(t-)$ for all $t\in\otime$.
\end{lemma}
A proof of this assertion can be found in~\cite[Lem.~2.11]{hyper}.
In the cited Lemma, the assertion is proved for a finite time horizon, but the reasoning also works on $[0,\infty)$ since one may restrict to a finite time, either $t$ in~\eqref{ineq2}
or the domain of the function $\phi$ in~\eqref{ineq1}.
\begin{lemma}\label{lem:initial}
Let $\V$ be a Banach space, $\E:\V\ra [0,\infty)$ strictly convex coercive  and $\Y$ be a Banach space such that $ \Y \subset^d \dom \E^* $. 
    Let $A$, $B\in \dom \E$ such that $ \langle A-B,\Phi \rangle = 0$ for all $\Phi \in \Y$. Then, it holds that $ A=B$. 
\end{lemma}
\begin{proof}
In order to derive a contradiction, we assume that $ A \neq B$. 
Due to the density and the fact that $\dom \partial \E^* $ is dense in $\dom \E^*$~\cite[Cor.\,2.38]{barbu}, we find two sequences $ \{\Phi^n_B\}_{n\in\N} \subset \dom \partial \E ^*$ and $\{\Phi^n_A\}_{n\in\N} \subset \dom \partial \E ^*$ such that 
$ \Phi^n_A \ra  \xi_A$ and $ \Phi^n_B\ra \xi_B$ in $\V^*$ with $ \xi_A\in \partial \E(A)$ and $\xi_B\in\partial \E(B)$. 
This implies  due to $A \neq B$ and the strict convexity of $\E$ that
\begin{align}
    0 < \E( A) + \E(B) - 2 \E\left ( \frac{A+B}{2}\right)\leq \frac{1}{2}  \langle \xi_A - \xi_B , A-B\rangle  = \lim _{n\ra \infty} \frac{1}{2}\langle \Phi^n_A  - \Phi^n_B , A - B \rangle  =0 \,
\end{align}
which is a contradiction and so we conclude $ A = B$. 
\end{proof}

\section{General existence result}
\label{general}
In this section, we consider a general nonlinear evolution equation of the form 
\begin{equation}
\t \f U(t)+A(t, \f U(t)) = 0 \quad \text{ in }\Y^* \text{ with } \quad\f U(0) = \f U_0 \in \V \,.\label{eq:gen}
\end{equation}

\begin{hypothesis}\label{hypo:1}
%
Let  $\Y  \subset \V $ be two Banach spaces such that $  \Y \subset^d \V^*$.
Let $ \V$  be  reflexive Banach space and $\E: \V \ra [0,\infty]$ be a strictly convex, lower semi-continuous, superlinear functional  on $\V$, \textit{i.e.}, $ \lim_{\|\f v\|_{\V}\ra \infty} \frac{\E(\f v)}{\|\f v\|_{\V}}= \infty$. Let $\diss : [0,\infty )\times \V \ra [-C_{\Psi}(t),\infty]$ be a mapping such that $  \dom \diss(t,\cdot) \subset \dom \E$ is convex, $ D\E^*(\f y)\in \dom\diss(t,\cdot)  $ for all  $\f y \in \Y $ and a.e.~$t\in(0,\infty)$.
Here $C_{\Psi}:[0,\infty)\ra[0,\infty)$ is  such that $\int_0^\infty C_{\Psi}(t)\de t<\infty$.
 Let $ A : [0,\infty)\times  \dom \Psi \ra \Y^\ast$ 
such that 
\begin{equation}\label{Adiss}
\langle A (t, D \E^*(\Phi)) , \Phi \rangle = \diss(t,D\E^*(\Phi))
\end{equation}
for all $ \Phi \in \Y $, a.e.~$t\in[0,\infty)$ and $\diss \circ D\E^\ast$ is continuous on $\Y$ with $ \E(\f U_{\min}) = 0 = \diss(t,\f U_{\min})$, being $\f U_{\min}$ the minimizer of $\E$. 
Both $\diss$ and $A$ are assumed to be measurable with respect to the first variable.
Finally, we assume that there exists a convex continuous function $\mathcal{K}:\Y  \ra [0,\infty) $ such that the mapping
\begin{equation}
\f U \mapsto \Psi (t,\f U) - \langle A(t,\f U) , \Phi \rangle+ \mathcal{K}(\Phi) \E(\f U)\label{ass:convex}
\end{equation}
defined on $\dom \diss$  is convex and lower semi-continuous for every $\Phi \in \Y  $ and a.e.~$t\in[0,\infty)$.
\end{hypothesis}
\begin{remark}[Hypothesis~\ref{hypo:1}]

    We note that the condition~\eqref{Adiss} formally assures that the system~\eqref{eq:gen} fulfills the energy-dissipation-mechanism~\eqref{intreoEnGen}. 
    Indeed, formally, we may test~\eqref{eq:gen} by $\partial \E(\f U)$, which leads by the Fenchel equivalences~\eqref{eq:fenchel} and the chain rule formula to 
    $$
    0=\langle \t \f U(t) , \partial \E(\f U(t)) \rangle + \langle A(t,\f U(t)) , \partial \E(\f U(t))\rangle = \t \E(\f U(t)) + \diss(t,\f U(t))\,.
    $$
\end{remark}

\begin{definition}\label{def:envar}
We call a pair $ ( \f U , E)\in L^\infty(\otime;\V)\times \BVl$ an \textit{energy-variational solution} to~\eqref{eq:gen} if $ \E(\f U) \leq E $ a.e. on $(0,\eT)$ and if 
 \begin{equation}  \label{envar}
          \left[ E - \langle \f U , \Phi\rangle \right]\Big|_s^t + \int_s^t \langle \f U, \t \Phi \rangle 
  +
   \diss(\tau,\f U) - \langle A(\tau, \f U) , \Phi\rangle + \mathcal{K}(\Phi) \left [\E(\f U) -E\right ]\de \tau \leq 0 
  \end{equation}
 for all $\Phi \in \C^1(\cltime;\Y) $ and
for a.e. $s<t\in\otime$ including $s=0$ with $\f U(0) =\f U_0$.
\end{definition}
\begin{theorem}\label{thm:envar}
For every $\f U_0\in\dom \E$, there exists an \textit{energy-variational solution} in the sense of Definition~\ref{def:envar} with $ E(0+)=\E(\f U_0)$.
\end{theorem}
\begin{remark}
   We note that every \textit{energy-variational solution} with $E = \E(\f U)$ on an interval $[t_0,t_1]\subset \cltime$ is a weak solution. 

   
   Indeed, we infer from~\eqref{envar} multiplied by $\alpha>0 $ and with $\Phi=\frac{1}{\alpha}\Theta$, where $\Theta \in \C^1(\cltime;\Y)$ and $E=\E(\f U)$ that
   \begin{align*}
    \left[\alpha\E(\f U)- \langle \f U , \Theta \rangle \right]\Big|_s^t+ \int_s^t \left[ \langle\f U , \t \Theta \rangle +\alpha \diss(\tau, \f U) - \langle A(\tau, \f U) , \Theta\rangle 
    \right] \de \tau \leq 0 \,
\end{align*}
for all $s<t\in[t_0,t_1]$ and all $\Theta \in \C^1(\cltime;\Y)$. As $\alpha \ra 0 $, we infer the weak formulation~\eqref{introweak} with an inequality. But for $\Phi=-\frac{1}{\alpha}\Theta$ the inequality with the opposite sign is inferred such that we showed the weak formulation~\eqref{introweak}.

\end{remark}
From the formulation~\eqref{envar}, we can even read off a higher regularity of the solution such that the relations are even fulfilled everywhere instead of almost everywhere in time. 
\begin{corollary}\label{cor:reg}
    Let $(\f U, E)$ be an \textit{energy-variational solution} in the sense of Definition~\ref{def:envar}. Then there exists $( \tilde{\f U},\tilde{E}) \in \C_w(\cltime;\V)\times \BVl$ such that $( \tilde{\f U},\tilde{E}) = (\f U, E)  $ a.e. in $\otime$ and the inequality $ \tilde{E}(t) \geq \E(\tilde{\f U}(t))$ holds true for all $t\in[0,\infty)$ and 
    \begin{align*}
    \left[\tilde E- \langle \tilde{\f U} , \Phi \rangle \right]\Big|_{s-}^{t+}+ \int_s^t \left[ \langle\tilde{\f U} , \t \Phi \rangle + \diss(\tau, \tilde{\f U}) - \langle A(\tau, \tilde{\f U}) , \Phi\rangle + \mathcal{K}(\Phi) \left[ \E ( \tilde{\f U}) -\tilde E \right]
    \right] \de \tau \leq 0 \,
\end{align*}
holds for all $\Phi \in \C^1(\cltime;\Y) $ and
for all $s<t\in\otime$. 

If, additionally, it holds that $E(s+) = \E(\f U(s))$ and $\E$ is $\rho$-uniformly convex, \textit{i.e.,} there exists a strictly monotone increasing function $\rho : [0,\infty) \ra [0,\infty)$ with $\rho(0) =0 $ such that $$ 2 \E \left( \frac{1}{2}(\f U + \f V)\right) \leq  \E( \f U ) + \E( \f V)  - \rho(\| \f U-  \f V \|_{\V}) $$ for all $ \f U$, $\f V\in\dom\E$, then $\f U$ is even right-hand continuous at $s$, \textit{i.e.,} $\lim_{t\searrow s} \f U (t) = \f U(s) $ in $ \V$. 
\end{corollary}
\begin{proof}
    In the same way as in~\cite[Prop.~3.1]{hyper}, we infer the existence of functions $( \tilde{\f U},\tilde{E}) \in \C_w(\cltime;\Y^*)\times \BVl$ fulfilling the asserted inequalities. The higher regularity  $\tilde{\f U}\in\C_w(\cltime;\V)$ follows from Lemma~\ref{lem:initial}. 
Indeed, consider a sequence $\{t_n\}_{n\in\N}\subset [0,\infty)$ such that $t_n\ra t\in \cltime$. Form the boundedness, $\f U\in L^\infty(\otime;\V)$, we infer that the sequence $\{ \f U(t_n)\}_{n\in\N} $ admits a weakly converging subsequence in $\V$. The fact that $\f U \in \C_w(\cltime;\Y^*)$ allows to identify this limit in $\Y$ such that Lemma~\ref{lem:initial} and the uniqueness of weak limits also guarantee that $ \f U(t_n)\rightharpoonup \f U(t)$ in $\V$ for the whole sequence. 
   
    Now let $ \{t_n\}_{n\in\N}\subset (s,\infty)$ such that $ t_n\searrow s$, we infer from $\f U \in \C_w(\cltime;\V)$ that $ \f U(t_n)\rightharpoonup \f U(s)$. The monotonicity of $E$ and the weakly lower semi-continuity of $\E$ imply 
    \begin{align}
        E(s) \geq \lim_{t_n \searrow s} E(t_n) \geq \liminf_{n \ra \infty } \E(\f U(t_n)) \geq \E(\f U(s)) = E(s) \,
    \end{align}
    such that $ \lim _{n\ra\infty} \E(\f U(t_n)) = \E( \f U)$. From the weak convergence and the strict convexity of $\E$, we infer the strong convergence by 
    \begin{align*}
        0        & \leq  \lim_{n\ra\infty}  \rho\left( \left\| {\f U(t_n) - \f U(t)} \right\| _{\V}\right) 
        \\
        &\leq \liminf_{n\ra\infty}\left[ \E( \f U(t_n)) + \E( \f U(t)) - 2 \E\left( \frac{\f U(t_n) + \f U(t)}{2}\right) \right] 
        \\
        & \leq  \lim_{n\ra\infty} \left[ \E(\f U(t_n) ) - \E( \f U(t)) - \langle \partial \E( \f U(t)) , \f U(t_n)  - \f U(t) \rangle \right] = 0 \,, 
    \end{align*}
    where the first inequality is due to the $\rho$-uniform convexity of $ \E$, the second one follows from  the definition of the subdifferential, and the equality follows from the convergence of the energy and the weak convergence.   This proves the strong convergence. 
\end{proof}
\begin{remark}[Properties of solutions]\label{rem:semi}
\textit{Energy-variational solutions} enjoy the semi flow property. That means that for a solution on $ \cltime$ every restriction   on an interval $[s,t]$ for all $ s<t\in\cltime$ to the initial value $( \f U(s), E(s-))$ is again a solution. Moreover, if $ ( \f U^1,E^1)$ is an \textit{energy-variational solution} on $[r,s]$ to the initial value $(\f U^1(r),E^1(r-))$ and $ (\f U^2,E^2)$ is an \textit{energy-variational solution} on $[s,t]$ with initial value $(\f U^2(s),E^2(s-))$ with $ 0\leq r<s<t<\infty$ and $(\f U^1(s),E^1(s+))=(\f U^2(s),E^2(s-)) $ then the concatenation
    \begin{align}
        (\f U,E) : = \begin{cases}
            (\f U^1,E^1) & \text{ on }[r,s]\\(\f U^2,E^2) & \text{ on }[s,t]
        \end{cases}
    \end{align}
    is a solution on $[r,t]$ with initial value $(\f U(r),E(r-) )=(\f U^1(r),E^1(r-)) $.

    \textit{Energy-variational solutions} are not unique, they are even far from being unique. If $(\f U, E)$ is an \textit{energy-variational solution} in the sense of Definition~\ref{def:envar}, also $(\f U, E+h)$ is an \textit{energy-variational solution} for any $h\geq 0$. The set is therefore admittedly too large. But this peculiar non-uniqueness is ruled out by the condition $E(0)=\E(\f U_0)$ in Theorem~\ref{thm:envar}. Nevertheless, \textit{energy-variational solution} cannot be expected to be unique. However, the set of \textit{energy-variational solutions} to a given initial value is convex and weakly-star closed, which follows in the same way as in~\cite[Prop.~3.4]{hyper}. Moreover, for the example of incompressible Newtonian fluid dynamics the solution set is continuous in the Kuratowski sense. These can be seen as indicators that the solution set is amenable for additional selection criteria in order to single out a reasonable solution via minimization. 
    Already the proposed algorithm in~\eqref{eq:timedis} follows from this idea by selecting the time-discrete solution that minimizes the energy at the current time point. 
\end{remark}

In order to state the second main result of the paper concerning weak-strong uniqueness of solutions, we need to reinforce Hypothesis~\ref{hypo:1} by the following.
\begin{hypothesis}\label{hypo:2} In addition to Hypothesis~\ref{hypo:1}, we  
assume that $ \E\in\C^2( \inter \dom \E )$ and that there exist a  Banach space $ \W $ with $\Y \subseteq \W \subseteq \V$, $p\in[1,\infty)$, and a constant $C>0$ such that  $$
\int_0^T \| \f U \|_{\W}^p\de t \leq C \int_0^T \left(\Psi(t,\f U) +C_{\Psi} + (\E(\f U))^p+ 1\right) \de t 
\,$$
for all $T>0$ and for all $\f U \in \dom \Psi $, which implies that $\dom \Psi \subset \W\subset \V $. 
Additionally, we  assume that $ \langle A,\Phi \rangle   $ and $\diss $, are Gateaux differentiable for all $\Phi \in \Y$ and a.e. $t\in (0,T)$ 
with respect to the second variable.
    
\end{hypothesis}
Using the space $ \W$, we define $\mathbb Z\otimeT : = L^\infty(\otimeT;\V)\cap L^p(\otimeT;\W)$ and $\widetilde{\mathbb Z}\otimeT:= L^1(\otimeT;\V^*)\oplus L^{p'}(\otimeT;\W^*)$, where $p'$ denotes the conjugated exponent to $p $, \textit{i.e.}, $p'=\frac{p}{p-1}$.
\begin{remark}
    This space $\mathbb Z\otimeT $ is only one possibility to characterize additional regularity of the \textit{energy-variational solution}. Indeed, it is possible to choose a larger space $\mathbb Z\otimeT$ in certain examples. 
    One possibility is to assume that the space $\W$ admits a decomposition $ \W=\W_1\times \ldots\times \W_n$ such that $ \f U= ( \f U_1^T,\ldots , \f U^T_n)^T$. Then, we  assume that there exist $ \{ p_i \}_{i=\{1,\ldots,n\}}$ such that $ pi\in[1,\infty)$ and we may  replace the left hand-side of the condition in Hypothesis~\ref{hypo:2} by 
    $ \int_0^T \sum _{i=1}^n \|  \f U_i\|_{\W_i}^{p_i}\de t $. 
    This will already imply higher regularity in the first example below. 
\end{remark}
\begin{proposition}[Relative energy inequality]\label{prop:rel}
Let $ (\f U, E) \in L^\infty(\otime;\V) \times \BVl  $  be an \textit{energy-variational solution} in the sense of Definition~\ref{def:envar} and let Hypothesis~\ref{hypo:2} be fulfilled. Then it holds
\begin{multline}\label{relenin}
    \mathcal{R}(\f U , E| \tU) \Big|_s^t + \int_s^t \mathcal{W}(\tau,\f U| \tU ) + \langle \t \tU + A(\tau,\tU) , D^2\E(\tU)(\f U-\tU) \rangle \de \tau 
   \\ \leq \int_s^t \mathcal{K}(D \E(\tU)) \mathcal{R}(\f U , E| \tU)  \de \tau 
    \end{multline}
    for all $s<t\in\cltimeT$ for all $\tU \in \C^0(\cltimeT;\V)\cap L^p(0,T;\W)$ such that
    \begin{align}\label{abstractreg}
    \begin{split}
         D\E(\tU)&\in L^\infty(\otimeT;\Y) \,,\\
         D\diss(t,\tU)&\in \widetilde{\mathbb Z}\otimeT\,,\\
         DA(t,\tU)&\in  \mathcal L ( \mathbb{Z}\otimeT ; L^1(\otimeT;\Y^*))\,, \\
         \t \tU,\,  A(t,\tU)&\in \widetilde{\mathbb Z}\otimeT\,,\\
         D^2\E(\tU)&\in  \mathcal{L}(\mathbb{Z}\otimeT,\mathbb{Z}\otimeT)\,,
             \end{split}
    \end{align}
    where $T\in(0,\infty]$.
Here we denote the relative energy by    \begin{align*}
\mathcal{R}(\f U , E| \tU) &:= E - \E(\tU) - \langle D\E( \tU) , \f U - \tU\rangle 
\intertext{and the relative dissipation via}
\mathcal{W}(t,\f U| \tU) &: = \diss(t,\f U) - \diss(t,\tU) - \langle D \diss(t,\tU), \f U- \tU\rangle \\&\quad - \langle A(t,\f U)-A(t,\tU)-DA(t,\tU)(\f U-\tU) , D\E(\tU)\rangle 
    \\
   & \quad + \mathcal{K}(D \E(\tU))\left[\E( \f U) - \E( \tU) - \langle D \E( \tU) ,\f U - \tU \rangle\right] \,.
    \end{align*}
   \end{proposition}
\begin{remark}
    We note that, due to the convexity of the energy $\E$ and the condition $E\geq \E(\f U)$, the relative energy 
    \begin{align*}
        \mathcal{R}(\f U , E| \tU) &:= [E-\E(\f U)] + [\E(\f U) - \E(\tU) - \langle D\E( \tU) , \f U - \tU\rangle]
    \end{align*}
    is non-negative and vanishes iff $ E =\E(\f U)$ and $\tU=\f U$. 
    Moreover, due to assumption~\eqref{ass:convex}, the relative dissipation~$\mathcal W$ is nonegative, since it is the subdifferential of the convex function in~\eqref{ass:convex}. 

    Furthermore, the condition $D\E(\tU)\in\Y$ implies that $ \mathcal K(D\E(\tU))$ is well defined. 
    Finally, we observe from Hypothesis~\ref{hypo:1} that $D\E^*(D\E(\tU))\in\dom\diss$. 

    Note that the Gateaux-derivatives of $\diss$ and $A$ are also well defined due to Hypothesis~\ref{hypo:2}.
    
\end{remark}
   
   \begin{corollary}[Weak-strong uniqueness]
   \label{cor:weakstrong}
   Let $ (\f U, E) 
   $  be an \textit{energy-variational solution} in the sense of Definition~\ref{def:envar} with $E(s)=\E(\f U(s))$ for one $s\in[0,\infty)$ and let Hypothesis~\ref{hypo:2} be fulfilled. Let $\tU \in \C^0(\cltimeT;\V)\cap L^p(0,T;\W)$
   be a strong solution, \textit{i.e.}, a function fulfilling the equation~\eqref{eq:gen}  in $\widetilde{\mathbb Z}\otimeT $  with $ \f U(s) = \tU(s)$ and $s<T$, enjoying the additional regularity of~\eqref{abstractreg},  then it holds that 
   \begin{align}
       \f U(t) =\tU(t) \quad \text{for all }t\in[s,T]\,.
   \end{align}
   \end{corollary}

\begin{remark}
    We note that the above corollary is stronger than the usual weak-strong uniqueness results, since the generalized solution coincides with the strong solution as soon as both coincide at one point in time. Usually these results are only formulated for $s=0$ such that the solutions only coincide, if they coincide in the initial value. Moreover, inequality~\eqref{relenin} gives a continuous dependence results for solutions on the initial value as long as the strong solution exists. 
\end{remark}
\begin{remark}[Comparison to dissipative solutions]
Under Hypothesis~\ref{hypo:2} every \textit{energy}-\textit{varia\-tio\-nal} \textit{solution} is a dissipative solution in the sense of Lions~\cite[Sec.~4.4]{lionsfluid}. Indeed, from inequality~\eqref{relenin}, we infer by a version of Gronwall's inequality and estimating $ E\geq \E(\f U) $ and inserting $ E_0=\E(\f U_0)$ as in the result of Theorem~\ref{thm:envar}  that 
\begin{multline*}
    \mathcal{R}(\f U(t) , \E(\f U(t)) | \tU(t)) \\ + \int_0^t \left[\mathcal{W}(\tau,\f U| \tU ) + \langle \t \tU + A(\tau,\tU) , D^2\E(\tU)(\f U-\tU) \rangle\right ]\e^{\int_s^t\mathcal{K}(D \E(\tU))\de \tau  }  \de s 
   \\ \leq  \mathcal{R}(\f U_0 , \E(\f U_0)| \tU(0) )  \e^{\int_0^t\mathcal{K}(D \E(\tU))\de s  } \,,
    \end{multline*}
     for all $\tU \in \C^0(\cltimeT;\V)\cap L^p(0,T;\W)$ fulfilling~\eqref{abstractreg}, 
    which is the usual dissipative formulation according to Lions. 
\end{remark}

\begin{proof}[Proof of Theorem~\ref{thm:envar}]
We divide the proof into several steps. 

\textit{Step 1: equidistant time discretization.}
Let  $N\in\N$ such that $ \mathcal{K}(0) < N $, we define $\tau :=1/N$, and we set $t^n:= \tau n$ for $n\in \N$
to obtain 
an equidistant partition of $\cltime$.
For convenience, we define $ \Psi^n : \V \ra [0,\infty )$ via $ \Psi^n (\f U) := \frac{1}{\tau}\int_{t^{n-1}}^{t^n} \diss (t, \f U) \de t $, $A^n : \dom \diss ^n \ra \Y^* $ via $ A^n (\f U) := \frac{1}{\tau}\int_{t^{n-1}}^{t^n} A ( t,\f U) \de t $, and $ C_{\Psi}^n = \frac{1}{\tau}\int_{t^{n-1}} ^{t^n} C_{\Psi}(t)\de t$. 
We set $\f U^0:=\f U_0\in \dom \E$ and define 
\begin{align*}
\mathcal{F}^n(\f U | \Phi) :={}& \E (\f U )  - \E(\f U^{n-1}) + \left ( \E (\f U )  - \E(\f U^{n-1}) -  \tau C_{\Psi}^n\right )  \tau \mathcal{K}(\Phi) 
\\&- \left \langle \f U - \f U^{n-1} , \Phi \right \rangle+ \tau \diss^n(\f U) - \tau \langle A^n(\f U) , \Phi \rangle \,.
\end{align*}
The set $\mathbb D^n:= \{ \f U \in \V \mid \E(\f U)+\tau \diss^n(\f U)  \leq \E(\f U^{n-1}) \}$ is convex and weakly compact in $\V$. Indeed, the convexity follows from condition~\eqref{ass:convex} with $\Phi=0$ and all $t\in[t^{n-1},t^n)$ and the fact that $\tau \mathcal{K}(0)<1$ such that $  \E(\f U)+\tau \diss^n(\f U) = \tau(\diss^n(\f U) + \mathcal{K}(0)\E(\f U)) + (1-\tau\mathcal{K}(0))\E(\f U) $. Since $ (1-\tau\mathcal{K}(0))> 0$ and $\E$ is coercive, the set $\mathbb D^n $ is even weakly-compact.
We define the iterate for $n\geq 1$ for a given $\f U^{n-1}$ via
\begin{equation}\label{eq:timedis}
\begin{aligned}
  \f U^n= \argmin _{\f U \in \mathbb D^n }  \, &\sup _{ \Phi \in \Y
  } \!
\mathcal{F}^n(\f U | \Phi) \,.
\end{aligned}
\end{equation}
In the next steps, we show that this definition actually makes sense. 

\textit{Step 2: Min-max theorem.}
In order to show that 
\begin{equation}
\inf_{\f U \in \mathbb D^n
} \sup_{\Phi \in\mathbb{Y}} \mathcal{F}^n(\f U| \Phi  ) = \sup _{\Phi\in\mathbb{Y}} \inf_{
\f U \in \mathbb D^n
}\mathcal{F}^n(\f U| \Phi ) \,\label{minmax}
\end{equation}
we apply a min-max theorem.
Recall that
 the set $\mathbb D^n$ is convex and weakly compact in $\V$. 
The function 
$ \f U \mapsto \mathcal{F}^n  ( \f U | \Phi)$ is convex and weakly lower semi-continuous for every $ \Phi \in\Y $ due to Hypothesis~\ref{hypo:1}. 
Moreover, the function $ \Phi \mapsto \mathcal{F}^n(\f U| \Phi)$ is concave for all $ \f U \in \V$ with $\E(\f U)\leq \E(\f U^{n-1})+ \tau C_{\Psi}^n$
since $\mathcal K$ is convex.
Therefore, \eqref{minmax} follows from Fan's 
min-max theorem~\cite[Theorem 2]{Fan1953}.

\textit{Step 3: Discrete energy variational inequality.} 
We want to prove the inequality
\begin{equation}
    \inf_{\f U \in \mathbb D^n 
} \sup_{\Phi \in\mathbb{Y}} \mathcal{F}^n (\f U| \Phi  ) \leq 0\,.\label{dis:envar}
\end{equation} 
From \textit{Step 2}, we infer the equality~\eqref{minmax}. In order to show~\eqref{dis:envar}, we consider $\Phi 
\in \Y $ arbitrary and $\tilde{\f U}:= D \E^*(\Phi) \in \dom \E$. We define $\hat{\f U}: = D\E^*(\alpha \Phi)$ with $\alpha =1 $ if $\E(\tU) +\tau \diss^n(\tU)< \E(\f U^{n-1})$ and with $\alpha \in (0,1]$ such that $\E(\hat{\f U}) +\tau \diss^n(\hat{\f U})= \E(\f U^{n-1})$, if $\E(\tU) +\tau \diss(t^n ,\tU)\geq \E(\f U^{n-1})$. We can always find such a value $\alpha\in(0,1]$ since the function 
$$f:[0,1]\ra [0, \E(\tU) +\tau \diss^n(\tU)]; \quad f(\alpha)= \E(D\E^*(\alpha \Phi)) +\tau \diss^n( D\E^*(\alpha \Phi))$$ is continuous with $f(0)=0$ and $f(1)=\E(\tU) +\tau \diss^n( \tU) \geq \E(\f U^{n-1})$ due to the intermediate value theorem. The continuity of $\diss \circ D\E^*$ follows from Hypothesis~\ref{hypo:1} and the continuity of $\alpha \mapsto\E \circ D\E^*(\alpha\Phi)$ from convex analysis. 
Indeed, via Fenchel's equivalences, we may write 
$\E \circ D\E^*(\alpha\Phi) = \langle D \E^*(\alpha\Phi) , \alpha \Phi\rangle - \E^*(\alpha \Phi) $ and from the assumptions on $\E$, we infer the continuity 
of $\E^*$ with $\dom\E^\ast=\V^*$~\cite[Prop.~2.25 and Thm.~2.14]{barbu}.
The mapping $\partial \E^*$ is single-valued and demi-continuous~\cite[Thm.~5.20]{roubicek}, which implies the continuity of $f$. 
Since $\f U_{\min}$ is the strict minimizer of $\E$, it holds $ \f 0 \in \partial \E( \f U_{\min })$ and thus by Fenchel's equivalences $\f U_{\min} \in \partial  \E ^*(\f 0)$. Since the subdifferential is indeed single-valued, due to the superlinearity and strict convexity  of $\E$~\cite[Prop.~2.47]{barbu} and~\cite[Thm.~5.20]{roubicek}, we infer $ \{ \f U_{\min}\} = D \E ^*(\f 0)$, which implies due to Hypothesis~\ref{hypo:1} that $f(0)=0$. 

With this choice of $\hat{\f U}$, we observe that $ \inf_{\f U \in \mathbb D^n} \mathcal{F}^n(\f U , \Phi) \leq{} \mathcal{F}^n(\hat{\f U}, \Phi)$ for which, we find
\begin{align*}
     \mathcal{F}^n(\hat{\f U}, \Phi) ={}& [ \E(\hat{\f U}) - \E(\f U^{n-1})]+ \tau\mathcal{K}(\Phi) [ \E(\hat{\f U}) - \E(\f U^{n-1})-{\tau}  C_{\Psi}^n ] + \tau \diss^n(\hat{\f U})\\&  - \frac{1}{\alpha}\left[ 
    \langle \hat{\f U} - \f U^{n-1}, \alpha \Phi \rangle + \tau \langle A^n(\hat{\f U}), \alpha \Phi \rangle 
    \right]
    \\
    \leq {}& \tau\mathcal{K}(\Phi) [ \E(\hat{\f U}) - \E(\f U^{n-1})-\tau C_{\Psi}^n] + \left(1- \frac{1}{\alpha}\right) \left[ 
    \E(\hat{\f U}) + \tau \diss^n( \hat{\f U}) - \E(\f U^{n-1}) 
    \right]
    \\
    ={}& \tau\mathcal{K}(\Phi) [ \E(\hat{\f U}) - \E(\f U^{n-1})-\tau C_{\Psi}^n]\leq -\tau^2\mathcal{K}(\Phi) [  C_{\Psi}^n+\diss^n(  \hat{\f U})] \leq 0 
    \,,
\end{align*}
where we used~\eqref{Adiss} and the convexity of $\E$ to infer the inequality and the choice of $\alpha$ to infer the second equation. 
The choice of $\hat{\f U}$ and the lower bound on $\diss$ allow to deduce the two last inequalities.

\textit{Step 4: Well defined optimization problem.} 
Let us define $\mathcal{H}(\f U) = \sup_{\Phi \in \Y} \mathcal{F}^n(\f U|\Phi)$. As the supremum of convex lower semi-continuous functions, $\mathcal{H}$ is a convex lower semi-continuous function on the convex and weakly compact set $\mathbb D^n$, $\mathcal{H}$  is even strictly convex due to the strict convexity of $\E$. 
From the previous two steps, we infer that $\mathcal{H}$ is proper, \textit{i.e.,} not equal to $+\infty$ everywhere on $\mathbb D^n$. 
Thus, the minimization problem $\min_{\f U\in \mathbb D^n} \mathcal{H}(\f U)$ has a unique minimizer and so definition~\eqref{eq:timedis} makes sense. 

\textit{Step 5: Prolongations.}
For functions $\phi \in \C^\infty_c([0,\eT))$ and $ \Phi \in \C^1( \cltime;\mathbb Y )$,
we define $ \phi^n := \phi(t^n)$, $ \Phi^n := \Phi (t^n)$, and $E^n := \E(\f U^n)$  for $n \in\N$.
Inserting $\Phi = \Phi ^{n-1}$ in~\eqref{dis:envar},
multiplying the resulting inequality by $\phi^{n-1}$ and summing this up for $n\in \N$ provides 
\[
\begin{aligned}
\sum_{n=1}^\infty&\left [ \phi^{n-1} ( E^n-E^{n-1} ) - \phi^{n-1} \langle \f U^n-\f U^{n-1} , \Phi^{n-1} \rangle \right ]
\\
&+ \tau \sum_{n=1}^\infty \phi^{n-1} \left [  \Psi^n(\f U^n) - \langle A^n(\f U^n) , \Phi^{n-1}\rangle + \mathcal{K}(\Phi ^{n-1} ) (\E(\f U^n) -E^{n-1}- \tau C_{\Psi}^n) \right ] \leq 0 \,.
\end{aligned}
\]
Since there exists a $n_0\in\N$ such that $\phi^n=0$ for all $n\geq n_0$,
we may use  a discrete integration-by-parts formula 
and divide by $\tau>0$, 
to obtain
\begin{equation}
\begin{aligned}
-\sum_{n=1}^\infty &\left [\frac{\phi^n- \phi^{n-1}}{\tau} [ E^n- \left \langle  \f U^n,\Phi^{n-1}\right \rangle]  - \phi^{n} \left \langle \f U^n ,\frac{ \Phi^n- \Phi^{n-1}}{\tau}\right  \rangle \right ]- \phi^0 [{ \E(\f U_0)-\langle \Phi^0, \f U_0\rangle }]
\\
&+ \sum_{n=1}^\infty \phi^{n-1} \left [ \diss^n(\f U^n)- \langle A^n(\f U^n), \Phi^{n-1})\rangle  + \mathcal{K}(\Phi ^{n-1} ) (\E(\f U^n) -E^{n-1}- \tau C_{\Psi}^n) \right ] \leq 0 \,.
\end{aligned}
\label{disrel}
\end{equation}

We define the piece-wise constant prolongations
\[
\begin{aligned}
\ov{\f U}^N(t) &:= \begin{cases}
\f U^n & \text{for } t \in (t^{n-1},t^n], \\
\f U_0 & \text{for } t = 0,
\end{cases} \,
\\
\ov {E}^N(t) &:= \begin{cases}
\mathcal{E}(\f U^{n})
& \text{for } t \in (t^{n-1},t^n],
 \\
\mathcal{E}(\f U_0) &\text{for } t = 0,
\end{cases}\,
\qquad
\un{E}^N(t) := 
\mathcal{E}(\f U^{n-1} ) 
&
 \text{for } t \in [t^{n-1},t^n).
\end{aligned}
\]
for all $n\in \mathbb{N}$.
Analogously, for  test functions $ \psi \in \C^1(\cltime; \mathbb{X})$,
where $\mathbb{X}$ is $\R$ or $\mathbb{Y}$,
we define the piece-wise constant and piece-wise linear prolongations by
\begin{align*}
    \overline{\psi}^N(t)& := \begin{cases}
\psi(t^n) & \text{for } t \in (t^{n-1},t^n], \\
\psi(0) & \text{for } t = 0,
\end{cases} \,\qquad
      \underline{\psi}^N(t) := 
\psi(t^{n-1}) \quad\text{for } t \in [t^{n-1},t^n),
\\
     \widehat{\psi}^N(t)& := \frac{\psi(t^n)-\psi(t^{n-1})}{\tau}(t-t^{n-1}) + \psi(t^{n-1})
      \quad \text{for } t\in[t^{n-1},t^n]\,
\end{align*}
for all $n\in\N$. 
Inserting this notation, the discrete energy-variational inequality \eqref{disrel}
becomes
\begin{multline}\label{disenin}
- \int_0^\infty \left[{ \t \widehat{\phi} ^N\left [ \ov E ^N - \langle \ov {\f U}^N, \un{\Phi}^N  \rangle \right ] - \ov \phi^N \langle \ov{\f U}^N  , \t \widehat{\Phi}^N\rangle  + \un{\phi}^N \mathcal{K}(\un{\Phi}^N )(\un{E}^N+\tau C_{\Psi}(t ))}\right] \de t 
\\
+ \int_0^\infty \un\phi^N \left [ \diss (t ,\ov{\f U}^n)  - \langle A(t ,\ov{\f U}^n) ,   {\un{\Phi}^N}\rangle + \mathcal{K}(\un{\Phi}^N )  \E(\ov{\f U}^N) \right ]\de t - \phi(0)\left[\E(\f U_0) - \langle \Phi(0) , \f U_0 \rangle \right] \leq 0\,. 
\end{multline}

\textit{Step 6: Convergence.}
Since we have 
$ \ov{E}^N+\tau \diss^n( \ov{\f U}^N )\leq \un{E}^N$ such that $ E^n - \tau C_{\Psi}^n \leq E^{n-1} $
fulfilling
$$
\sum_{i=1}^k | E^i-E^{i-1}| 
\leq E^0-E^k+2 \tau\sum_{i=1}^k C_{\Psi}^i \leq E(0) + 2 \int_0^\infty C_{\Psi}(t) \de t \,. 
$$
 we obtain that $t\mapsto \ov{E}^N(t)  $ and  $t\mapsto \un{E}^N(t)  $ are non-negative
 and  bounded in $\BVl$.

Moreover, by the coercivity of $\E $ and the $L^\infty$ bound of $\ov{E}^N$, we infer 
from $ \E ( \ov{\f U}^N(t)) \leq \ov{E}^N(t) $ 
 that the sequence $\{ \ov{\f U}^N\}_{N\in\N}$ is bounded in $L^\infty(\otime;\V)$.
Thus, we may extract (not-relabeled) subsequences 
 such that there exist $\ov{E},\un{E}\in \BVl;$ and $\f U \in L^\infty(\otime;\V)$ such that
\[
\begin{aligned}
\ov{\f U}^{N} &\xrightharpoonup{\ast} \f U 
&&\quad\text{in } L^\infty(\otime;\V)\,,\\
( \ov{E}^N,\,\un{E}^N) &\xrightharpoonup{\ \ast\ } (\ov{E},\, \un{E}) 
&&\quad \text{in }\BVl\,,
\end{aligned}
\]
where we used the weak$^*$ compactness of functions of bounded variation in $\BVl$~\cite[Proposition 3.13]{BV}.
We  show next that $\ov{E}^N $ and $\un{E}^N$ converge to the same limit,
that is, $\ov{E}=\un{E}$ a.e.~in $\otime$. 
Due to the estimate~$E^n  \leq E^{n-1}+\tau C_{\Psi}^n $, we find
\begin{multline*}
\int_0^\eT | \ov{E}^N-\un{E}^N|\de t = \sum_{n=1} ^\infty \tau | E^n-E^{n-1}| \leq \sum_{n=1} ^\infty \tau\left( | E^n-E^{n-1}- \tau C_{\Psi}^n| + \tau C_{\Psi}^n\right)  \\= \sum_{n=1} ^\infty \tau\left( E^{n-1}- E^n+ \tau C_{\Psi}^n + \tau C_{\Psi}^n\right)  = 
\sum_{n=1}^\infty \tau   (E^{n-1}-E^{n} + 2\tau  C_{\Psi}^n)  
\\= \tau( E(0) - \lim_{n\ra \infty}E^n) + 2\tau  \int_0^\infty C_{\Psi} (t) \de t \\ \leq\tau  \left ( E(0)+  2\tau  \int_0^\infty C_{\Psi} (t) \de t\right ) \longrightarrow  0 \quad 
\text{ as } N \to\infty\,,
\end{multline*}
where the first equality follows from the definition of $\ov E^n $ and $\un E^N$, the first inequality is the triangle inequality, the subsequent equality follows from the fact that the occurring term in the absolute value has a sign, and the last inequality follows from the non-negativity of the values $E^n$. 
This allows to identify $\un{E}=\ov{E}=:E$ a.e.~in $\otime $.  
On the discrete Level, it holds that $\ov{E}^N(t)=\E(\ov{\f U}^N(t) $ for all $t\in[0,\infty)$. 
For any $ \zeta\in\C^\infty_c([0,\infty))$ with $\zeta(t)\geq 0$ for all $t\in \cltime$, we observe that 
\begin{align*}
    \int_0^\infty \zeta(t) E (t)\de t = \lim_{N\ra\infty} \int_0^\infty \zeta(t)\ov{E}^N(t)\de t = \lim_{N\ra\infty}\int_0^\infty \zeta(t) \E(\ov{\f U}^N(t)) \de t \geq \int_0^\infty \zeta(t) \E(\f U(t)) \de t\,,
\end{align*}
where the inequality holds due to the weakly lower semi-continuity of $\E$ and Fatou's Lemma. 
This implies that $ E \geq \E(\f U)$ a.e. in $\otime$. 
Since $\phi$ and $\Phi$ are continuously differentiable on $\otime$, we have
\begin{align*}
\t \widehat{\phi}^N &\ra \t \phi, 
& \ov{\phi}^N &\ra \phi, 
& \un{\phi}^N &\ra \phi  
&&\text{ pointwise in } \cltime \text{ as }N \ra \infty \,,
\\
\t \widehat{\Phi}^N &\ra \t \Phi, 
& \un{\Phi}^N &\ra \Phi, 
& \nabla\un{\Phi}^N &\ra \nabla\Phi \quad &\text{in } \Y 
&\text{ pointwise in } \cltime \text{ as }N \ra \infty 
\,.
\end{align*}
With these observations, we may pass to the limit in the weak form~\eqref{disenin}. 
We note that $\ov{\f U}^N$  occurs linearly in the first line of~\eqref{disenin}. All other terms are bounded and converge almost everywhere in $(0,\infty)$.
This implies that 
\begin{multline*}
-\lim_{N\ra\infty} \int_0^{\eT} { \t \widehat{\phi} ^N\left [ \ov E ^N - \langle \ov {\f U}^N , \un{\Phi}^N \rangle \right ] - \ov \phi^N \langle \ov{\f U}^N  , \t \widehat{\Phi}^N\rangle  + \un{\phi}^N \mathcal{K}(\un{\Phi}^N )(\un{E}^N +\tau C_{\Psi}(t )) }\de t \\
=- \int_0^\eT { \t {\phi}\left [  E  - \langle   {\f U} , {\Phi} \rangle \right ] -  \phi \langle {\f U}  , \t {\Phi}\rangle  + {\phi} \mathcal{K}({\Phi} ){E}} \de t\,.
\end{multline*}
The term $\tau \int_0^\infty \un{\phi}^N \mathcal{K}(\un{\Phi}^N ) C_{\Psi}(t )\de t $ can be bounded by $\tau $ times a constant such that is vanishes as $\tau \ra 0$.
Observing that the second line in~\eqref{disenin} is bounded from below due to Hypothesis~\ref{hypo:1}, since every lower semi-continuous function on a bounded set is bounded from below, and that $\phi\geq 0$ in $\cltime $, we may apply Fatou's lemma and the weakly-lower semi-continuity of the function from~\eqref{ass:convex} as well as the continuity of $\mathcal{K}$ in order to pass to the limit in the second line of~\eqref{disenin}, which yields
\begin{multline*}
\liminf_{N\ra\infty} \left [ \int_0^\eT \un\phi^N \left [ 
\diss(t ,\ov{\f U}) + \langle A(t , \ov{\f U}) , 
{\un{\Phi}^N}\rangle + \mathcal{K}(\un{\Phi}^N )  \E(\ov{\f U}^N) \right ]\de t \right ] \\
\geq  \int_0^\eT\liminf_{N\ra\infty} { \un\phi^N \left [ \diss(t ,{\ov{\f U}^N}) - \langle A (t ,{\ov{\f U}^N}),   {\un{\Phi}^N}\rangle + \mathcal{K}(\un{\Phi}^N )  \E(\ov{\f U}^N) \right ]}\de t \\
 \geq \int_0^T \phi\left[ \diss (t,\f U)- \langle A(t,\f U) , \Phi\rangle+  \mathcal{K}(\Phi  ) \mathcal{E}(\f U)\right]   \de t\,.
\end{multline*}
Summarizing, we infer in the limit from \eqref{disenin} that
\[
\begin{aligned}
-\int_0^\eT &\t  \phi \left [ {E} - \left \langle {\f U } ,  \Phi \right \rangle\right ]  \de t - \phi(0)\left[\E(\f U_0) - \langle  \f U_0,\Phi(0)  \rangle \right]\\
&+\int_0^\eT \phi  \left [\left \langle {\f U } , \t  \Phi \right \rangle+ \diss(t,\f U) - \langle A (t,\f U),   {{\Phi}}\rangle  +    \mathcal{K}(\Phi ) [\E({\f U}) - E]   \right ] \de t 
 \leq 0\,.
\end{aligned}
\]
Via Lemma~\ref{lem:invar}, we now end up with the energy-variational 
inequality~\eqref{envar} and with
\begin{equation*}
\lim_{t \searrow 0}\left [ E (t) - \langle \f U(t),\Phi(t)  \rangle \right] \leq \E(\f U_0) - \langle\f U_0, \Phi(0)  \rangle\,,
\end{equation*}
after possibly redefining the function on a set of measure zero.
Choosing $\Phi\equiv 0$, we find $E(0+)\leq \E(\f U_0)$ such that $ E(0)=\E(\f U_0)$. Multiplying the above relation by $\alpha>0$ and choosing $\Phi = \alpha^{-1}\tilde \Phi$ for $\tilde\Phi\in\Y$ we find $ \lim_{t \searrow 0}\langle \f U(t),\tilde \Phi \rangle \leq \langle \f U_0 , \tilde \Phi \rangle $ and the same for $\tilde \Phi $ replaced by $ - \tilde \Phi$. That implies
 $\langle\f U(0+)-\f U_0, \Phi\rangle=0 $ for all $\Phi \in \mathbb Y$. From Lemma~\ref{lem:initial}, we infer that $\f U(0+)=\f U_0$ in $\V$. This concludes the proof of Thm.~\ref{thm:envar}.

\end{proof}
\begin{proof}[Proof of Proposition~\ref{prop:rel}]
Adding and subtracting the term $\langle \t \tU + A(t,\tU) , D^2\E(\tU)(\f U-\tU) \rangle$ from \eqref{envar}, {with $\Phi=D\E(\tU)$}, we find
\begin{align}\begin{split}
         0 \geq &  \left[E- \langle \f U , D\E(\tU) \rangle \right]\Big|_s^t+ \int_s^t \left[ \langle\f U , \t D\E(\tU) \rangle - \langle \t \tU ,  D^2\E(\tU)(\f U-\tU)\rangle \right] \de \tau \\
     &
     + \int_s^t\left[  \diss(\tau,\f U) - \langle A(\tau,\f U) , D\E(\tU)\rangle -\langle A(\tau,\f U) , D^2\E(\tU)(\f U-\tU) \rangle
    \right] \de \tau 
    \\&+\int_s^t \langle \t \tU + A(\tau,\f U) , D^2\E(\tU)(\f U-\tU) \rangle + \mathcal{K}(\Phi) \left[ \E ( \f U) - E \right]\de \tau\,.
\end{split}\label{calcrel}
\end{align}
From the main theorem of differential and integral calculus, we observe
\begin{multline}\label{calcderive}
    \int_s^t \left[ \langle\f U , \t D\E(\tU) \rangle - \langle \t \tU ,  D^2\E(\tU)(\f U-\tU)\rangle \right] \de \tau  \\= \int_s^t  \langle \t \tU ,  D^2\E(\tU) \tU\rangle\de \tau  =- \int_s^t \t \left[ \E(\tU) - \langle D\E(\tU), \tU\rangle \right]\de \tau. 
\end{multline}
    From~\eqref{Adiss}, we find $ \langle A (t,\tU) , D \E(\tU)\rangle = \diss( t,\tU)$ for all $ \tU\in\Y$.  Taking the derivative of relation~\eqref{Adiss}, we infer 
    \begin{align}\label{chainrule}
      \langle D\diss(t,\tU), \xi\rangle =   \langle  A (t,\tU) , D^2 \E(\tU)\xi \rangle + \langle D A (t,\tU) \xi , D \E(\tU)\rangle \,. 
    \end{align}
    Inserting~\eqref{calcderive} 
    and the above two relations 
     into~\eqref{calcrel}%
, we conclude
    \begin{align}\begin{split}
         0 \geq & \mathcal R (\f U, E|\tU) \Big|_s^t
     + \int_s^t\left[  \diss(\tau,\f U) - \diss(\tau,\tU)-\langle D \diss(\tau,\tU), \f U-\tU\rangle \right]\de \tau  \\ &-\int_s^t\left[
    \langle A(\tau,\f U)- A(\tau,\tU)-DA(\tau,\tU)(\f U-\tU) , D\E(\tU)\rangle + \mathcal{K}(D \E(\tU)) \left[ \E ( \f U) - E \right]
    \right] \de \tau 
    \\&+\int_s^t \langle \t \tU + A(\tau,\f U) , D^2\E(\tU)(\f U-\tU) \rangle \de \tau\,
\end{split}
\end{align}
which implies the assertion. 
\end{proof}
\section{Overview of viscoelastic models}
In this section we introduce a brief presentation of different incompressible viscoelastic models available in the literature. Our aim is to apply the abstract theoretical framework introduced in Section \ref{general} to prove the existence of \textit{energy-variational solutions} of viscoelastic models, without the need to introduce, as typically done in the literature, diffusive regularizations in the transport equations for the kinematic variables to obtain the existence of weak solutions. Throughout the discussion, we will indicate the model variants which constitutively satisfy the assumptions introduced in Section \ref{general}, and thus which are fitted by the general theoretical framework, while we will highlight the limits which forbid its application to some other model variants.
\newline
The starting point for the derivation of viscoelastic models is the definition of a kinematic variable $\mathbb{G}:\Omega\times [0,\infty)\rightarrow \mathbb{R}^{d\times d}$, which is a mapping for infinitesimal deformations between the initial and the current configurations of a viscoelastic body. In Eulerian coordinates, given the velocity field $\f v:\Omega\times [0,\infty)\rightarrow \mathbb{R}^{d}$,  the kinematic variable $\mathbb{G}$ satisfies an hyperbolic equation of type
\begin{equation}
    \label{model1}
    D_t\mathbb{G}=h(\mathbb{G},\nabla \f v),
\end{equation}
where $D_t\mathbb{G}:= \partial_t \mathbb{G}+ \f v \cdot \nabla \mathbb{G}$ is the material derivative in the Eulerian coordinates and $h:\mathbb{R}^{d\times d}\times \mathbb{R}^{d\times d}\rightarrow \mathbb{R}^{d\times d}$ is a function representing the time variation of the mapping associated to infinitesimal deformations, depending on both $\mathbb{G}$ and $\nabla \f v$.
\newline
The model equations can be derived from the principle of virtual powers, which gives the conservation equations for the linear and angular momenta expressed in terms of the kinematic
variable $\mathbb{G}$ and of a stress tensor $\mathbb{T}$, which is the power-conjugated variable to $\nabla \f v$. The form of the stress tensor is constitutively assigned, in terms of the kinematic variable and of the velocity field, in order for the system to satisfy the Clausius--Duhem inequality. Specifically,  considering an arbitrary part of the material $R(t)\subseteq \Omega$, moving with the material, the principle of virtual powers states that the sum of the virtual power of acceleration forces plus the virtual power of internal forces expended within $R(t)$ is equal to the virtual power expended in $R(t)$ by material external to $R(t)$ or by external forces (see e.g. \cite{fremond,gurtin}). Moreover, as a consequence of frame-indifference, the internal power expended within $R(t)$ for rigid virtual velocities must be equal to zero. Defining the set of virtual velocities $\mathcal{V}:=\{\hat{\f v}:\Omega\times \cltime\rightarrow \mathbb{R}^{d}, \hat{\f v}={\f v}_d \; \text{on} \; \partial \Omega_d\times \cltime\}$, where $\partial\Omega_d\subseteq \partial \Omega$, we may define the virtual power of acceleration and internal forces as

\begin{equation}
\label{model1vp}
p_{\text{acc}}(R(t),\hat{\f v}):=\int_{R(t)}D_t \f v\cdot \hat{\f v}{\de x}, \quad p_{\text{int}}(R(t),\hat{\f v}):=\int_{R(t)}\mathbb{T} : \nabla \hat{\f v}{\de x} ,
\end{equation}
and the virtual power of external forces as
\[
p_{\text{ext}}(R(t),\hat{\f v}):=\int_{R(t)}\f f\cdot \hat{\f v}{\de x} +\int_{\partial R(t)}\f t_n\cdot \hat{\f v}{\de S},
\]
where $\f f$ is an external field, while $\f t_n$ is a traction on the boundary of $R(t)$. Then, the principle of virtual powers becomes
\[
p_{\text{acc}}(R(t),\hat{\f v})+p_{\text{int}}(R(t),\hat{\f v})=p_{\text{ext}}(R(t),\hat{\f v}), \quad \forall \hat{\f v}\in \mathcal{V}, \forall R(t)\subseteq \Omega,
\]
which implies, together with the kinematical constraint \eqref{model1}, that the following system of coupled PDEs is satisfied
\begin{align}
    \label{model2}
     D_t \f v -\diver \mathbb T&=\f f, \quad \diver \f v=0,\\
     D_t\mathbb{G}&=h(\mathbb{G},\nabla \f v),
\end{align}
endowed with the boundary conditions
\begin{equation}
    \label{model3}
    {\f v}={\f v}_d \;\; \text{on} \; \partial \Omega_d\times \cltime, \quad \mathbb T\f n=\f t_n \;\;\text{on} \;\; \partial \Omega_n\times \cltime,
\end{equation}
with $\partial \Omega_d\cup \partial \Omega_n=\partial \Omega$, $|\partial \Omega_d\cap \partial \Omega_n|=0$. We observe that we are considering only Dirichlet and Neumann boundary conditions for ease of exposition, but also other boundary conditions, like Robin, could be allowed in the present theoretical framework.
Moreover, defining the set of rigid virtual velocities as $\mathcal{V}_{\text{rigid}}:=\{\hat{\f v}:\Omega\times \cltime\rightarrow \mathbb{R}^{d}:\hat{\f v}=\f v_0+\mathbb{A} \f x\}$, for any constant $\f v_0\in \mathbb{R}^d$ and constant $\mathbb{A}\in \mathbb{R}_{skw}^{d\times d}$, the fact that $p_{\text{int}}(R(t),\hat{\f v})=0$ for any $\hat{\f v}\in \mathcal{V}_{\text{rigid}}$, $\forall R(t)\subseteq \Omega$, implies that $\mathbb{T}\in \mathbb{R}_{sym}^{d\times d}$. Assuming the energy density of the system to be of the form
\begin{equation}
\label{model3e}
e(\mathbb{G}, \f v)=\frac{1}{2}|\f v|^2+\tilde{e}(\mathbb{G}),
\end{equation}
i.e. given by the sum of the kinetic energy density and of the elastic free energy density $\tilde{e}$, the Clausius--Duhem inequality takes the form \cite{gurtin}
\begin{equation}
    \label{model4}
    \frac{d}{dt}\int_{R(t)}e(\mathbb{G}, \f v)\leq p_{\text{acc}}(R(t),{\f v})+p_{\text{int}}(R(t),{\f v})+\int_{R(t)}p\diver \f v, 
\end{equation}
where $p$ is the Lagrange multiplier for the incompressibility constraint $\diver \f v=0$. 
Applying the Reynolds transport theorem and using \eqref{model3e} and \eqref{model1vp}, the Clausius--Duhem inequality becomes
\begin{align}
    &\notag \int_{R(t)}\f v \cdot D_t\f v+\int_{R(t)}\frac{\partial \tilde{e}}{\partial \mathbb{G}}:D_t \mathbb{G}+\int_{R(t)}e\diver \f v=\\
    &\label{model4bis} \leq \int_{R(t)}\f v \cdot D_t\f v+\int_{R(t)}\mathbb{T}:\nabla \f v+\int_{R(t)}p\diver \f v,
    \end{align}
Specific forms for the stress tensor $\mathbb{T}$ are then consitutively assigned in order for \eqref{model4bis} to be satisfied.
\newline
Different choices for the kinematic variable $\mathbb{G}$, together with different constitutive assumptions for the form of the free energy density $\tilde{e}(\mathbb{G})$, lead to different viscoelastic models. In the following, we list different variants, associated to different choices of the kinematic variable, which have been considered in the literature.
\begin{itemize}
    \item \textbf{Kelvin--Voigt viscoelasticity.} A standard approach is to consider as kinematic variable {$\mathbb{G}$} the deformation gradient $\mathbb{F}$ (i.e. the Jacobian of the deformation map between the initial and the current configuration) (see e.g. \cite{hu}), which satisfies the transport equation
    \begin{equation}
    \label{model5}
    D_t \mathbb{F}=\nabla \f v \mathbb{F}.
    \end{equation}
    Considering \eqref{model5} in \eqref{model4bis}, we obtain that
    \[
    \int_{R(t)}\frac{\partial \tilde{e}}{\partial \mathbb{F}}\mathbb{F}^T:\nabla \f v\leq \int_{R(t)}\mathbb{T}:\nabla \f v+\int_{R(t)}(p-e)\diver \f v,
    \]
    hence a general constitutive assumption for the stress tensor which verifies inequality \eqref{model4} is
    \begin{equation}
    \label{model6}
    \mathbb{T}=2\mu (\nabla \f v)_{sym}-p\mathbb I+\frac{\partial \tilde{e}}{\partial \mathbb{F}}\mathbb{F}^T=:\mathbb{T}_{visc}+\mathbb{T}_{el},
    \end{equation}
    where we renamed $p\leftarrow p-e$, and $\mu>0$ is the fluid viscosity of the material. Here, $\mathbb{T}_{visc}:=\mu (\nabla \f v)_{sym}-p\MI$ is the viscous contribution to the stress tensor and $\mathbb{T}_{el}=\frac{\partial \tilde{e}}{\partial \mathbb{F}}\mathbb{F}^T$ is the elastic contribution.
   
    Another common possibility is to consider as kinematic variable the left Cauchy--Green tensor $\mathbb{B}=\mathbb{F}\mathbb{F}^T\in \mathbb{R}_{sym}^{d\times d}$, which satisfies, as a consequence of \eqref{model5}, the transport equation
    \begin{equation}
    \label{model7}
    D_t \mathbb{B}=\nabla \f v \mathbb{B} + \mathbb{B} (\nabla \f v)^T,
    \end{equation}
    which is a form of the Oldroyd-B equation \cite{oldroyd}.
    Considering \eqref{model7} in \eqref{model4bis}, analogously to \eqref{model6}, a general constitutive assumption for the stress tensor in terms of the variable $\mathbb{B}$
    which verifies inequality \eqref{model4} is
    \begin{equation}
    \label{model8}
    \mathbb{T}_{el}=2\frac{\partial \tilde{e}}{\partial \mathbb{B}}\mathbb{B}.
    \end{equation}
    We also cite the possibility to describe viscoelasticity through the kinematic variable given by the elastic stress tensor \eqref{model8} (see e.g. \cite{malek2}).
    \begin{remark}
        \label{fb}
        As reported in \cite{perrotti}, given the formula $\mathbb{B}=\mathbb{F}\mathbb{F}^T$, the most general transport equation for the deformation gradient $\mathbb{F}$ which implies the Oldroyd-B equation \eqref{model7} is
        \begin{equation}
    \label{model9}
    D_t \mathbb{F}=\mathbb{F} \mathbb{W}+\nabla \f v \mathbb{F},
    \end{equation}
    for any $\mathbb{W}\in \mathbb{R}_{skw}^{d\times d}$.
    \end{remark}
    In \cite{perrotti} the following general form of the transport equation for $\mathbb{F}$ is considered instead of \eqref{model5}:
    \begin{equation}
        \label{model10}
        D_t\mathbb{F}=a\nabla \f v \mathbb{F}+b(\nabla \f v)^T\mathbb{F},
    \end{equation}
    with $a,b\in \mathbb{R}$. This corresponds to non-standard elasticity where the strain measure $\mathbb{F}$ does not necessarily respond as the Jacobian of the deformation, but it may take into account incompatible components coming from the molecular theory of viscoelasticity. Frame indifference implies that there exists $\alpha \in \mathbb{R}$ such that $a=\frac{\alpha+2}{4}$, $b=\frac{\alpha-2}{4}$ \cite{segalman}.
    Considering \eqref{model10} in \eqref{model4bis}, analogously to \eqref{model6}, a general constitutive assumption for the stress tensor which verifies inequality \eqref{model4} is
    \begin{equation}
    \label{model11}
    \mathbb{T}_{el}=(a+b)\frac{\partial \tilde{e}}{\partial \mathbb{F}}\mathbb{F}^T=\frac{\alpha}{2}\frac{\partial \tilde{e}}{\partial \mathbb{F}}\mathbb{F}^T.
    \end{equation}
    As a consequence of \eqref{model10}, the left Cauchy--Green tensor $\mathbb{B}=\mathbb{F}\mathbb{F}^T$ satisfies the transport equation
    \begin{equation}
    \label{model12}
    D_t \mathbb{B}=a\left(\nabla \f v \mathbb{B}+\mathbb{B} (\nabla \f v)^T \right) + b\left((\nabla \f v)^T\mathbb{B}+\mathbb{B} \nabla \f v\right)=2\left((\nabla \f v)_{skw}\mathbb{B}\right)_{sym}+\alpha\left((\nabla \f v)_{sym}\mathbb{B}\right)_{sym}.
    \end{equation}
   Considering \eqref{model12} in \eqref{model4bis}, analogously to \eqref{model6}, a general constitutive assumption for the stress tensor in terms of the variable $\mathbb{B}$
    which verifies inequality \eqref{model4} is
    \begin{equation}
    \label{model13}
    \mathbb{T}_{el}=2(a+b)\frac{\partial \tilde{e}}{\partial \mathbb{B}}\mathbb{B}=\alpha \frac{\partial \tilde{e}}{\partial \mathbb{B}}\mathbb{B}.
    \end{equation}
    \item \textbf{Generalized viscoelasticity with stress relaxation.} We now introduce viscoelastic models which consider the phenomenon of stress relaxation, which is a process resulting from the collective microscopic dissipative phenomena in a viscoelastic material modeling living tissues. We refer the interested reader to \cite{malek2}, where a general framework to describe stress relaxation is obtained by expressing the deformation gradient $\mathbb{F}$ as a multiplicative decomposition in terms of a dissipative component and an elastic component.
    A first variant of \eqref{model10} found in literature (cf., e.g., \cite{perrotti}) is the transport equation 
    \begin{equation}
        \label{model17}
     D_t\mathbb{F}=a\nabla \f v \mathbb{F}+b(\nabla \f v)^T\mathbb{F}-\frac{1}{2\mu_p}\frac{\partial \tilde{e}}{\partial \mathbb{F}},   
    \end{equation}
    where $\mu_p>0$ is a relaxation constant related to the material viscosity. The latter transport equation is the one considered in \cite{perrotti}, which adds to \eqref{model10} a dissipative contribution proportional to $\frac{\partial \tilde{e}}{\partial \mathbb{F}}$. 
    Considering \eqref{model17} in \eqref{model4bis}, analogously to \eqref{model6}, we obtain the general constitutive assumption 
    \begin{equation}
    \label{model15}
    \mathbb{T}_{el}=(a+b)\frac{\partial \tilde{e}}{\partial \mathbb{F}}\mathbb{F}^T=\frac{\alpha}{2}\frac{\partial \tilde{e}}{\partial \mathbb{F}}\mathbb{F}^T.
    \end{equation}
    Moreover, the dissipative contribution in \eqref{model17} gives rise to the term
    \[
    \tilde{\Psi}:=\frac{1}{2\mu_p}\int_{R(t)}\frac{\partial \tilde{e}}{\partial \mathbb{F}}:\frac{\partial \tilde{e}}{\partial \mathbb{F}}
    \]
    in the inequality \eqref{model4}, where $\tilde{\Psi}$ is the dissipation associated to the elastic deformation.
    Another variant of \eqref{model10} found in literature is the transport equation 
    \begin{equation}
        \label{model19}
     D_t\mathbb{F}=a\nabla \f v \mathbb{F}+b(\nabla \f v)^T\mathbb{F}-\frac{1}{2\mu_p}\mathbb{F}\mathbb{F}^{T}\frac{\partial \tilde{e}}{\partial \mathbb{F}},   
    \end{equation}
    where the dissipative contribution is proportional to the elastic free energy derivative up to transformation via the metrics $\mathbb{F}\mathbb{F}^{T}$. The latter contribution in the transport equation gives rise to the dissipation term
    \[
    \tilde{\Psi}:=\frac{1}{2\mu_p}\int_{R(t)}\mathbb{F}^{T}\frac{\partial \tilde{e}}{\partial \mathbb{F}}:\mathbb{F}^{T}\frac{\partial \tilde{e}}{\partial \mathbb{F}}.
    \]
    \newline
    As a consequence of \eqref{model17}, the left Cauchy--Green tensor $\mathbb{B}=\mathbb{F}\mathbb{F}^T$ satisfies the transport equation
    \begin{align}
    \label{model23}
    D_t \mathbb{B}=&a\left(\nabla \f v \mathbb{B}+\mathbb{B} (\nabla \f v)^T \right) + b\left((\nabla \f v)^T\mathbb{B}+\mathbb{B} \nabla \f v\right)-\frac{1}{\mu_p}\frac{\partial \tilde{e}}{\partial \mathbb{B}}\mathbb{B}\\
     \notag =&{2}\left((\nabla \f v)_{skw}\mathbb{B}\right)_{sym}+\alpha\left((\nabla \f v)_{sym}\mathbb{B}\right)_{sym}-\frac{1}{\mu_p}\frac{\partial \tilde{e}}{\partial \mathbb{B}}\mathbb{B}.
    \end{align}
    The latter transport equation may be seen as a generalized Oldroyd-B type equation \cite{oldroyd}.
    Considering \eqref{model23} in \eqref{model4bis}, analogously to \eqref{model6}, we obtain the general constitutive assumption for the stress tensor 
    \begin{equation}
    \label{model21}
    \mathbb{T}_{el}=2(a+b)\frac{\partial \tilde{e}}{\partial \mathbb{B}}\mathbb{B}=\alpha \frac{\partial \tilde{e}}{\partial \mathbb{B}}\mathbb{B},
    \end{equation}    
    and the dissipation term 
    \begin{equation}
    \label{model23d}
    \tilde{\Psi}:=\frac{1}{\mu_p}\int_{R(t)}\frac{\partial \tilde{e}}{\partial \mathbb{B}}:\frac{\partial \tilde{e}}{\partial \mathbb{B}}\mathbb{B}.
    \end{equation}
    The transport equation for the left Cauchy--Green tensor associated to \eqref{model19} is
    \begin{align}
    \label{model25}
    D_t \mathbb{B}=&a\left(\nabla \f v \mathbb{B}+\mathbb{B} (\nabla \f v)^T \right) + b\left((\nabla \f v)^T\mathbb{B}+\mathbb{B} \nabla \f v\right)-\frac{1}{\mu_p}\mathbb{B}\frac{\partial \tilde{e}}{\partial \mathbb{B}}\mathbb{B}\\
     \notag =&{2}\left((\nabla \f v)_{skw}\mathbb{B}\right)_{sym}{+}\alpha\left((\nabla \f v)_{sym}\mathbb{B}\right)_{sym}-\frac{1}{\mu_p}\mathbb{B}\frac{\partial \tilde{e}}{\partial \mathbb{B}}\mathbb{B}.
    \end{align}
    The latter transport equation may be seen as a generalized Giesekus type equation \cite{giesekus}.
    The dissipative contribution in \eqref{model25} gives rise to
    the dissipation term 
    \begin{equation}
    \label{model25d}
    {\tilde{\Psi}}:=\frac{1}{\mu_p}\int_{R(t)}\mathbb{F}^T\frac{\partial \tilde{e}}{\partial \mathbb{B}}\mathbb{F}:\mathbb{F}^T\frac{\partial \tilde{e}}{\partial \mathbb{B}}\mathbb{F}\de x.
    \end{equation}
\end{itemize}
We now specify different variants of viscoelastic models, associated to different choices of the kinematic variable $\mathbb{G}$, given constitutive assumptions for the form of the free energy density $\tilde{e}(\mathbb{G})$. The standard elastic free energy density for Oldroyd-B and Giesekus type models is of the Neo--Hookean form, i.e.
\begin{equation}
    \label{model28}
    \tilde{e}(\mathbb{B})=\tr (\mathbb{B}-\mathbb I - \ln (\mathbb{B})).
\end{equation}
Collecting \eqref{model2}, \eqref{model21}, \eqref{model23} and \eqref{model28} we obtain the following generalized {\em Oldroyd-B model}:
 \begin{subequations}
\label{ob1}     
 \begin{align}
\t \f v + ( \f v \cdot \nabla ) \f v + \nabla p - \mu \Delta \f v  - \alpha  \di  \mathbb{B} ={}& \f f \,, \quad \di \f v = 0 \,,\label{ob1v}\\
\t \mathbb{B} +  ( \f v \cdot \nabla ) \mathbb{B} -{2}[ (\nabla \f v)_{\skw} \mathbb{B}]_{\sym} - \alpha   [  ( \nabla \f v )_{\sym} \mathbb{B}]_{\sym}  + \frac{{1}}{\mu_p}(\mathbb{B} - \mathbb I) ={}& \MO \,, \quad  ( \mathbb{B})_{\skw} = \MO \, .\label{ob1B}
\end{align}
\end{subequations}
Collecting \eqref{model2}, \eqref{model21}, \eqref{model25} and \eqref{model28} we instead obtain the following {\em generalized Giesekus model}:
 \begin{subequations}
\label{gk1}     
 \begin{align}
\t \f v + ( \f v \cdot \nabla ) \f v + \nabla p - \mu \Delta \f v  - \alpha  \di \mathbb{B} ={}& \f f \,, \quad \di \f v = 0 \,,\label{gk1v}\\
\t \mathbb{B} +  ( \f v \cdot \nabla ) \mathbb{B} -{2}[ (\nabla \f v)_{\skw} \mathbb{B}]_{\sym} - \alpha   [  ( \nabla \f v )_{\sym} \mathbb{B}]_{\sym}  + \frac{{1}}{\mu_p}(\mathbb{B}^2 - \mathbb{B}) ={}& \MO \,, \quad  ( \mathbb{B})_{\skw} =\MO \, .\label{gk1B}
\end{align}
\end{subequations}
We observe that both the models \eqref{ob1} and \eqref{gk1} do not satisfy  Hypothesis~\ref{hypo:1} of the abstract theoretical framework, and thus the existence of \textit{energy-variational solutions} to these models cannot be inferred from Theorem \ref{thm:envar}. Indeed, the elastic free energy \eqref{model28} is not superlinear, and thus the energy functional $\mathcal{E}$ does not satisfy the hypothesis in \ref{hypo:1}. Moreover, as observed in the Introduction, the \textit{a priori} bounds coming from the dissipative inequality \eqref{model4} do not imply the integrability of the quadratic terms in \eqref{ob1B}. 
\newline A natural regularization of the energy~\eqref{model28} 
was introduced in \cite{Malek} by adding a quadratic term leading to
\begin{equation}
    \label{model30}
    \tilde{e}(\mathbb{B})=(1-\beta)\tr (\mathbb{B}-\mathbb I - \ln (\mathbb{B}))+\frac{\beta}{2}|\mathbb{B} - \mathbb I|^2,
\end{equation}
with $0<\beta<1$. We note that \eqref{model30} is superlinear. In \cite{Malek}, a transport equation for $\mathbb{B}$ which takes into account both Oldroyd-B type and Giesekus-type relaxation processes was considered, leading to the model
 \begin{subequations}
\label{mal1}     
 \begin{align}
\t \f v + ( \f v \cdot \nabla ) \f v + \nabla p - \mu \Delta \f v  - \alpha  \di \left ( (1-\beta) (\mathbb{B}-I)  +\beta ( \mathbb{B}^2 - \mathbb{B}) \right ) ={}& \f f \,, \quad \di \f v = 0 \,,\label{mal1v}\\
\t \mathbb{B} +  ( \f v \cdot \nabla ) \mathbb{B} -{2}[ (\nabla \f v)_{\skw} \mathbb{B}]_{\sym} - \alpha   [  ( \nabla \f v )_{\sym} \mathbb{B}]_{\sym}  + \mathbb{B} -\mathbb  I + \delta ( \mathbb{B}^2 - \mathbb{B}) ={}& \MO \,, \quad  ( \mathbb{B})_{\skw} = \MO \, .\label{mal1B}
\end{align}
\end{subequations}
Note that for $\delta=0$ the transport equation \eqref{mal1B} is of Oldroyd-B type, while for $\delta>0$ it is of Giesekus type. Note also that the particular form of the relaxation terms in \eqref{mal1B} breaks the dissipative structure of the system if $\delta \neq \frac{\beta}{1-\beta}$, \textit{i.e.}, in this cases the dissipation can not be interpreted as a multiple of the derivative of the free energy with respect to $\bB$. In~\cite{Malek}, stress diffusion is added to the above model, which breaks the dissipative structure as well. We disregard stress diffusion here and prove existence of \textit{energy-variational solutions} to system~\eqref{mal1} in Theorem~\ref{thm:visco2} below. 
\newline

We conclude this section by observing that a possible way to overcome the impossibility of inscribing the generalized Oldroyd-B model \eqref{ob1} in our abstract theoretical framework could be to consider its counterpart expressed in terms of the kinematic variable $\mathbb{F}$, i.e. considering the Neo--Hookean elastic free energy density  
\begin{equation}
    \label{model31}
    \tilde{e}(\mathbb{F})=|\mathbb{F}|^2-|\MI|^2-\log(\det (\mathbb{F}^2)),
\end{equation}
which is superlinear. The drawback of this approach is that \eqref{model31} is not necessarily convex, since the tensor $\mathbb{F}$ is not symmetric, and hence it does not satisfy assumption \ref{hypo:1}. Indeed, we recall that the function $-\log(\det(\f X^2))=-2\log(\det(\f X))$ is convex on the set $\mathbb{R}_{sym,+}^{d\times d}$ (see e.g. \cite[Section 3.1.5]{boyd}), while it has no convexity properties on the set $\mathbb{R}_{+}^{d\times d}$.
Unlike the kinematic variable $\mathbb{B}$, which is symmetric as a solution of its transport equation if it is symmetric at the initial time, the transport equations \ref{model10}, \ref{model17} or \ref{model19} do not imply that their solutions have any symmetry. A possible way to proceed is to consider a transport equation for $\mathbb{F}$ of the type \eqref{model9}, i.e.
\begin{equation}\label{NoeHookOldroyd}
    D_t\mathbb{F}=\mathbb{F}\mathbb{W}+a \nabla \f v \mathbb{F}+b(\nabla \f v)^T\mathbb{F},
\end{equation}
choosing the skew tensor $\mathbb{W}$ in such a way that
\[
\left(\mathbb{F}\mathbb{W}+a \nabla \f v \mathbb{F}+b(\nabla \f v)^T\mathbb{F}\right)_{skw}=\MO,
\]
implying that the solution of the transport equation is symmetric if it is symmetric at the initial time. This approach has been employed in \cite{balci}. In the Appendix we will derive the general form for such a $\mathbb{W}$, giving a mechanical and geometrical interpretation to it as (minus) the angular velocity for the system (see \eqref{eqmodb3}), while correspondingly the symmetric solution $\mathbb{F}$ is the square root of $\mathbb{B}$ and contains information about the stretching in the system. 
We observe that, since $\mathbb{W}= \mathbb{W}(\mathbb{F},\nabla \f v)$  depends linearly on $\nabla \f v$ but non linearly on $\mathbb{F}$, it's difficult to fit the theoretical framework to the model obtained with this approach, checking that assumption \ref{hypo:1}, in particular the convexity of the mapping \eqref{ass:convex}, are verified. 
\newline
An alternative and more practicable way to proceed in this direction is to consider the symmetrized
transport equation of \eqref{model17}, i.e. 
\begin{equation}
\label{model32}
D_t \mathbb{F}=\frac{a}{2}\left(\nabla \f v \mathbb{F}+\mathbb{F} (\nabla \f v)^T \right) + \frac{b}{2}\left((\nabla \f v)^T\mathbb{F}+\mathbb{F} \nabla \f v\right)-\frac{1}{\mu_p}\left(\mathbb{F}-\mathbb{F}^{-1}\right),
\end{equation}
in which the solution preserves the symmetry of initial conditions.
We observe that \eqref{model32} corresponds to the generalization \eqref{model10} of the transport equation applied to a symmetric tensor $\mathbb{F}$ of the form (in components)
\[
\mathbb{F}_{ij}=\frac{1}{2}\left(\frac{\partial x_i}{\partial X_j}+\frac{\partial x_j}{\partial X_i}\right),
\]
with $\f x, \f X$ coordinates of the current and the initial configurations respectively. This means that the symmetric mapping $\mathbb{F}$ represents pure shear. Note that for pure shear the angular velocity $\mathbb{W}\equiv \MO$ and in this case the viscoelastic model becomes
 \begin{subequations}
\label{fsim2}     
 \begin{align}
\t \f v + ( \f v \cdot \nabla ) \f v + \nabla p - \mu \Delta \f v  - \alpha  \di \left ( \mathbb{F}^2  \right ) ={}& \f f \,, \quad \di \f v = 0 \,,\label{fsim2v}\\
\t \mathbb{F} +  ( \f v \cdot \nabla ) \mathbb{F} -2 [(\nabla \f v)_{\skw} \mathbb{F}]_{\sym} - \alpha [  ( \nabla \f v )_{\sym} \mathbb{F}]_{\sym}  + \frac{1}{\mu_p}  (\mathbb{F} - \mathbb{F}^{-1}) ={}& \MO \,, \quad  ( \mathbb{F})_{\skw} = \MO \, .\label{fsim2F}
\end{align}
\end{subequations}

The corresponding Oldroyd-B model to \eqref{fsim2}, which is formally derived from~\eqref{fsim2} by multiplying~\eqref{fsim2F} by $\bF$ and defining $\bB=\mathbb F^2$, is the model
 \begin{subequations}
\label{bsim2}     
 \begin{align}
\t \f v + ( \f v \cdot \nabla ) \f v + \nabla p - \mu \Delta \f v  - \alpha  \di \mathbb{B} ={}& \f f \,, \quad \di \f v = 0 \,,\label{bsim2v}\\
\t \mathbb{B} +  ( \f v \cdot \nabla ) \mathbb{B} -{2}[ (\nabla \f v)_{\skw} \mathbb{B}]_{\sym}- \alpha [  ( \nabla \f v )_{\sym} \mathbb{B}]_{\sym} \qquad\qquad \\  -\alpha \sqrt{\mathbb{B}}(\nabla \f v)_{sym}\sqrt{\mathbb{B}}+ \frac{{1}}{\mu_p}  (\mathbb{B} - \MI) ={}& \MO \,, \quad  ( \mathbb{B})_{\skw} = \MO \, .\label{bsim2B}
\end{align}
\end{subequations}
We want to remark that the described equivalence of the model is only formal. We can only show the existence of an \textit{energy-variational solution} to the system~\eqref{fsim2} in Theorem~\ref{thm:visco} and not the necessary additional regularity of this solutions  in order to infer the existence of an \textit{energy-variational solution} to the system~\eqref{bsim2}. Nevertheless, the solutions of the systems should exhibit similar phenomena. 
\section{Oldroyd-B model for viscoelastic fluid with regularized energy\label{sec:Mal}}
\label{models}
 In this section, we want to apply the abstract result contained in Theorem~\ref{thm:envar}  to system~\eqref{mal1} describing viscoelastic fluids with a regularized energy, which was  proposed in~\cite{Malek}. 
%
 %
 We repeat the system here for convenience,
  \begin{subequations}
\label{sysvisco2}     
 \begin{align}
\t \f v + ( \f v \cdot \nabla ) \f v + \nabla p - \mu \Delta \f v  - \alpha  \di \left ( (1-\beta) (\bF-\MI)  +\beta ( \bF^2 - \bF) \right ) ={}& \f f \,, \quad \di \f v = 0 \,,\label{eq1sym2}\\
\t \bF +  ( \f v \cdot \nabla ) \bF -2[ (\nabla \f v)_{\skw} \bF]_{\sym} - \alpha   [  ( \nabla \f v )_{\sym} \bF]_{\sym}  + \bF - \MI + \delta ( \bF^2 - \bF) ={}& \MO \,, \quad  ( \bF)_{\skw} = \MO \, ,\label{eq2sym2}
\end{align}
with $\beta\in (0,1)$, $\alpha \in \R\backslash \{0\}$, and$\delta,\,\mu \in(0,\infty)$.
This system  is equipped with initial and boundary conditions, 
\begin{align}
    \f v(0) = \f v_0 \,, \quad \bF(0) = \bF_0 \quad \text{in } \Omega\\
    \f v  = 0 \quad \text{on } \partial\Omega  \times \cltime\,. \label{eq:bound}
\end{align}
 \end{subequations}
\begin{remark}[Boundary conditions]
In comparison to the model proposed in~\cite{Malek}, we delete the stress diffusion and chose certain constants to be 1. Moreover, we choose homogeneous Dirichlet boundary conditions for the velocity field in comparison to the following Navier-slip boundary conditions in~\cite{Malek},
    \begin{align}
        \f v\cdot \f n  = 0\,, \quad - \f v _\tau= \left( \left( 2\mu(\nabla \f v )_{\sym} + \alpha ((1-\beta) (\bF-\MI)  +\beta (\bF^2-\bF) )  \right) \cdot \f n \right)_{\tau} \quad \text{on } \partial\Omega  \times \cltime\,.
    \end{align}
    The handling of these boundary conditions is somehow standard and we decided to simplify the model in this regard for the readers' convenience. However, the same results hold true with some obvious changes for the generalized formulation in terms of additional boundary terms and the underlying spaces.  
\end{remark}


We define the space $ \V $ to be given  by $\V := \Ha(\Omega)\times L^2_{\sym}(\Omega)$ and we define the energy $\E : \V \ra [0,\infty)$ via
\begin{align}\label{visco:energy2}
    \E(\f v , \bF) : = \begin{cases}
        \left[ \frac{1}{2}\lVert \f v\rVert^ 2_{L^2(\Omega)} +\int_\Omega   (1-\beta)  \tr(\bF-\MI-\log(\bF)) + \frac{\beta}2| \bF-\MI|^2  \de  x\right] & \text{if }\bF \in L^2_{\sym,+}(\Omega)\\
        \infty & \text{else}
    \end{cases} \,.
\end{align}
We define the space $\Y$ to be 
\begin{align}\label{spaceY}
\Y := 
(W^{1,\infty} \cap \Hsig)(\Omega; \R^d ) \times (L^\infty\cap W^{1,3})(\Omega; \R^{d\times d}_{\sym})
\,.
\end{align}


\begin{definition}\label{def:visco2}
We call the triple $$(\f v, \bF, E)\in L^\infty(\otime;\Ha(\Omega))\cap L^2(\otime;\Hsig(\Omega)) \times L^\infty(\otime;L^2_{\sym,+}(\Omega)) \times \BVl$$ an \textit{energy-variational solution} to \eqref{eq1sym2}--\eqref{eq1sym2} with the initial values $ \f v_0\in \Ha(\Omega)$ and $\bF_0\in L^2_{\sym,+}(\Omega) $, if $ \mathcal{E}(\f v, \bB) \leq E $ and
\begin{align*}
&\left[E - ( \f v , \vv) - ( \bF , \bft )  \right]\Big|_s^t -  \int_s^t 
\langle \f f , \f v - \vv \rangle \de \tau \\&+ \int_s^t ( \f v , \t \vv ) + ( \bF , \t \bft)  + ( \f v\otimes \f v- \mu \nabla \f v-\alpha ((1-\beta)(\bF - \MI )+ \beta  (\bF^2-\bF)  )  ; \nabla \vv) \de \tau &\\
&+ \int_s^t   (( \f v \cdot\nabla  )\bft ;\bF) + 2(  (\nabla \f v)_{\skw} \bF ;  \bft) +\alpha ( (\nabla \f v)_{\sym} \bF ; \bft) -  \left( \bF -\MI + \delta  (\bF^2- \bF) ; \sigma \right) \de \tau &
\\
&+ \int_s^t \mu ( \nabla \f v , \nabla \f v)+ \left( (1-\beta)  ( \MI - \bF^{-1}) + (\beta +\delta(1-\beta)) (\bF-\MI) + \delta \beta \bF (\bF-\MI); \bF-\MI
\right)
\de \tau  \\
&+ \int_s^t
 \mathcal{K}( \vv , \bft )  \left( \mathcal{E}(\f v , \bB ) - E \right) \de \tau &\leq 0 \,.
\end{align*}
holds for a.e.~$t>s\in(0,\infty)$ including $s=0$ with $\f v(0)=\f v_0$ as well as $ \bB(0)=\bB_0$ and all  $ (\vv , \bft)\in \C^1 (\cltime;\Y ) $,
where the regularity weight $\mathcal{K}$ is given by
\begin{align}\label{KB}
\begin{split}
    \mathcal{K}(\vv , \bft) &= 2\max\{1,\alpha\}\| (\nabla \vv)_{\sym,-} \|_{L^\infty(\Omega ; \R^{d\times d})} + \frac{C^2}{\beta \mu}\| \nabla \bft\|_{L^3(\Omega)}^2 \\&\quad+ \frac{(2+\alpha)^2}{
4\mu\beta}\| \bft\|_{L^\infty(\Omega)}^2 + \frac{2(\beta+\delta-3\delta\beta)_-}{\beta} \,.
\end{split}
\end{align}
\end{definition}
The first result regards existence of \textit{energy-variational solutions} for \eqref{eq1sym2}--\eqref{eq2sym2} as well as their weak-strong uniqueness.

\begin{theorem}\label{thm:visco2}
For every $ \f v _0 \in \Ha$, $ \bF_0 \in L^2_{\sym,+}$ with $\ln \det \bF_0  \in L^1(\Omega)$  and $ \f f \in L^2(\otime ; (\Hsig)^*)$  there exists an \textit{energy-variational solution} in the sense of Definition~\ref{def:visco2} with $ \E(\f v_0, \bF_0) = E(0)$.  

Let $ (\tv , \tbF)$ be a weak solution to~\eqref{sysvisco2} with 
\begin{align}\label{regStrongM}
\begin{split}
        \tv & \in L^\infty(\otimeT;W^{1,\infty}(\Omega)) 
        \cap C^1(\cltimeT; \Ha(\Omega))\\
    \tbF &\in L^\infty (\otimeT;(W^{1,3}\cap L^\infty)( \Omega; \R^{d\times d}_{\sym,+})) \cap C^1(\cltimeT; L^2_{\sym}(\Omega)) 
\end{split}
\end{align}
such that there exists a $b\in(0,1)$ with $ \det \tbF \geq b$ for a.e.~$(x,t)\in\Omega \times (0,T)$
as well as $ (\tv(0),\tbF(0))=(\f v_0,\bF_0) $. Then it holds that 
$$
\f v = \tv \quad\text{and}\quad \bF = \tbF \quad \text{ for all }(x,t)\in\Omega \times \otimeT\,.
$$
\end{theorem}
\begin{remark}
    We also want to stress that the assertions of Corollary~\ref{cor:reg} hold since $\E$ given in~\eqref{visco:energy2} fulfills the $\rho$-uniform convexity assumption with $\rho(x) = \frac{1}2{x^2}$. This also shows that the initial values are actually attained in the strong sense. 
\end{remark}
\begin{proof}
In order to prove Theorem~{\ref{thm:visco2}}, we want to apply Theorem~\ref{thm:envar}. Therefore, we have to show that Hypothesis~\ref{hypo:1} is fulfilled. 

As above, $\V := \Ha\times L^2_{\sym}$ and the energy is given by~\eqref{visco:energy2}.
The dissipation potential $\Psi : \V \ra [0,\infty]$ is given via the potential
$$\chi(\f v, \bF) := \mu\lvert \nabla \f v \rvert^2  + \tr\left(
(1-\beta) \bF  ( \MI - \bF^{-1})^2 + (\beta +\delta(1-\beta)) (\bF-\MI)^2 + \delta \beta \bF (\bF-\MI)^2\right) 
$$
such that
\begin{align}
    \label{visco:diss2}
    \diss(t,\f v , \bF) = \begin{cases}
        \int_\Omega \chi(\f v,\bF)  \de x - \langle \f f(t) , \f v \rangle & \text{if }\f v \in \Hsig \text{ and }\bF \in L^3_{\sym,+} \text{ with }\bF^{-1}\in L^1_{\sym,+}  \\
        \infty & \text{else}
    \end{cases}\,.
\end{align}
For the lower bound of the dissipation potential, we may choose $ C_{\Psi}(t) := \frac{1}{2\mu}  \| \f f \|_{(\Hsig)^*}^2$. Indeed, via Young's inequality we can estimate  
\begin{multline*}
    \Psi (t,\f v,\bF) \geq \mu \| \nabla \f v \|_{L^2}^2 - \langle \f f(t) , \f v \rangle \geq \mu \| \nabla \f v \|_{L^2}^2 - \| \nabla \f v \|_{L^2}\| \f f \|_{(\Hsig)^*} \\\geq \frac{\mu}{2} \| \nabla \f v \|_{L^2}^2  -\frac{1}{2\mu}  \| \f f \|_{(\Hsig)^*}^2 \geq    -C_{\Psi}(t) \,.
\end{multline*}
Note that the regularity of $\f f \in L^2(\otime;(\Hsig)^*)$ is essential for $\diss$ to be well-defined and lower semi-continuous.

In order to check the next assumptions, we need to calculate the convex conjugate of the energy $\E$. The subdifferential is single-valued on its domain and is given by  
\begin{align}\label{subEM}
\partial \E : \V \ra 2^{\V^*} \qquad
    \partial \E( \f v , \bF) =\begin{pmatrix}
        \f v \\ (1-\beta)(\MI - \bF^{-1}) + \beta ( \bF - \MI) 
    \end{pmatrix}\,.
\end{align}
 Its inverse gives the subdifferential of the convex conjugate according to Fenchels equivalences, which is  single-valued on $\V^*= \V $ and given by 
\begin{align}\label{subEstarM}
    \partial \E^* : \V^* \ra 2^{\V} \qquad
    \partial \E^*(\vv,\bft) = \begin{pmatrix}
        \vv\\
        \f B(\bft) 
    \end{pmatrix} \,
    \hbox{ when $\bft \in L^2_{\sym,+}$ and $(\vv , \f 0)^T$ else.} 
\end{align}
Here,  we used the definition 
\begin{equation}\label{Bdef}
    \f B ( \bft) =  \frac{\bft-(1-2\beta) \MI}{2\beta} + \sqrt{\frac{(\bft-(1-2\beta)\MI)^2}{4\beta^2}+ \frac{(1-\beta)}{\beta}\MI }\,,
\end{equation}
  which comes from the fact that  $\bft$
solves the matrix equation $\bft = (1-\beta)(\MI - \bF^{-1}) + \beta ( \bF -\MI) $ and as a consequence of the positive definiteness of $\bF$, it holds  $$ 0=\beta \bF ^2 - (\bft- (1-2\beta)\MI) \bF - (1-\beta) \MI \,.$$
  The plus sign in front of the square root in the definition of $\f B$ reflects the fact that $\partial \E^*$ should map into the domain of $\E$. 
We note that the derivative $D\E^*(\vv,\bft) \in \dom\diss$ for all $ (\vv,\bft)\in\Y$, which follows from the fact that $\vv \in \Hsig$ and $ \bft \in L^2_{\sym,+} $ implying that the second component of $ \partial \E^*$ belongs to $ L^2_{\sym,+}$.   

The operator $A:\dom\diss \ra \Y^*$ is given by all terms in the equations~\eqref{sysvisco2} despite the time derivative, such that
\begin{align}\label{Asys2}
\begin{split}
        \langle A(t,\f v, \bF), (\vv,\bft)\rangle = &
  \mu \langle \nabla \f v, \nabla \vv\rangle - \left( \f v \otimes \f v -\alpha ((1-\beta)(\bF - \MI )+ \beta  (\bF^2-\bF)  )  ; \nabla \vv\right)  \\ &- \langle \f f (t), \vv\rangle 
   - \left( \bF \otimes \f v \dreidotkom \nabla \bft\right) - \left( ( \nabla \f v )_{\skw} \bF; \bft\right)\\& - \alpha \left( ( \nabla \f v)_{\sym} \bF; \bft\right) + \left( \bF -\MI + \delta  (\bF^2- \bF), \bft\right) \,.
\end{split}
\end{align}
We need to show that~\eqref{Adiss} is fulfilled for $A$ and $\diss$ as given above, \textit{i.e.,}
$$
\left\langle A( t,D \E^*(\vv , \bft)) ; \begin{pmatrix}
\vv\\\bft
\end{pmatrix}\right \rangle = \Psi(t,D \E^*(\vv,\bft)).
$$
In order to verify this assumption, we insert $\f v= \vv$ and $\bF = \f B(\bft)$ in~\eqref{Asys2}. 
This calculation resembles the usual energy estimate. 
 Since, $\vv$ is a solenoidal vector field, the convection term vanishes, $ ( \vv \otimes \vv ;\nabla \vv)=0$. From the identity $$ 
  (1-\beta) (\f B(\bft) - \MI) +\beta ((\f B(\bft)) ^2- \f B(\bft)) - \bft\f B(\bft) =  \f 0\,,$$ we find that 
 \begin{multline*}
      \alpha (  ( 1-\beta)(\f B(\bft)- \MI) +\beta(\f B^2(\bft) - \f B(\bft))  ;\nabla \vv) - \alpha (( \nabla \vv )_{\sym} \f B(\bft); \bft)\\ = \alpha (   (1-\beta) (\f B(\bft) - \MI)+\beta (\f B^2(\bft) - \f B(\bft))  - \bft\f B(\bft); (\nabla \vv)_{\sym})=  0 \,.
 \end{multline*}
 Furthermore, we observe by an integration by parts and expressing $\bft$ in terms of $\f B(\bft)$ that
 \begin{align*}
 -& ( \f B(\bft)\otimes \vv \dreidotkom \nabla \bft) -2 ( ( \nabla \vv)_{\skw}\f B(\bft);\bft) 
 \\&= ( ( \vv \cdot \nabla ) \f B(\bft) ;(1-\beta) (\MI-\f B^{-1}(\bft)) + \beta( \f B(\bft) - \MI) ) \\& \quad -2 ( ( \nabla \vv)_{\skw}\f B(\bft) ; (1-\beta) (\MI-\f B^{-1}(\bft)) + \beta( \f B(\bft) - \MI) ) 
 \\&= \int_\Omega(1-\beta) ( \vv \cdot \nabla )\tr\left ( \f B(\bft) - \ln (\f B ( \bft))\right )+ \frac{\beta}{2}  ( \vv \cdot \nabla ) | \f B(\bft)- \MI|^2\de x \\ & \quad   -2\left( ( \nabla \vv)_{\skw} ; (1-\beta) (\f B(\bft) - \MI)+ \beta( \f B^2(\bft) - \f B(\bft))  \right) = 0\,.
 \end{align*}
 The first line on the right-hand side vanishes since $\vv$ is a solenoidal vector field and the second one since  it is a Frobenius product of a skew-symmetric and symmetric matrix. 
Using all these cancellations, we find  for the remaining dissipative terms
\begin{align*}
  \langle A(t,\vv, \f B(\bft)), (\vv,\bft)^T\rangle  ={} & \int_\Omega \chi(\vv, \f B(\bft)) \de x  - \langle \f f (t), \vv\rangle = \diss ( t,\vv , \f B(\bft))\,
\end{align*} 
 for all $ (\vv,\bft)\in \Y$. 
 Due to~\eqref{subEstarM}, this calculation implies that assumption~\eqref{Adiss} is fulfilled. 
 The continuity of the function $\f B$ in $\bft$ implies the continuity of the function $\Psi \circ D\E^*$  in the topology of $\Y$.
 Since  $\chi$ is a continuous function in $\bF$, the combination of continuous functions is continuous and the boundedness allows to deduce continuity of $\Psi \circ D\E^*$ by Lebegue's theorem on dominated convergence. 
We also  observe from \eqref{visco:energy2} and \eqref{visco:diss2} that at their global minimum $\E(\f v_{\min} , \bF_{\min})=\diss(\f v_{\min} , \bF_{\min})=0$ with $ \f v_{\min}= \f 0$ and $\bF_{\min} = \MI$. 

The final condition, we need to verify is~\eqref{ass:convex}. 
Therefore, we observe the following estimates
\begin{align*}
    \left( \f v \otimes \f v - \alpha\beta(\bF-\MI)^2 ; \nabla \vv\right)={}&
    \left( \f v \otimes \f v - \alpha\beta (\bF-\MI)^2 ; (\nabla \vv)_{\sym}\right)
    \\
    \leq{}& \left(\| \f v \|_{L^2(\Omega)}^2 + \alpha\beta \| \bF-\MI\|_{L^2(\Omega)}^2 \right) \| (\nabla \vv)_{\sym} \|_{L^\infty(\Omega ; \R^{d\times d})} 
 \\   \leq{}&  \| (\nabla \vv)_{\sym} \|_{L^\infty(\Omega ; \R^{d\times d})} 2\max\{1,\alpha\}\E(\f v ,\bF)
 \intertext{and}
  \left( \bF \otimes \f v \dreidotkom \nabla \bft\right)= \left( (\bF-\MI) \otimes \f v \dreidotkom \nabla \bft\right) \leq{}& \lVert \f v\rVert _{L^6(\Omega)} \sqrt{\beta}\| \bF-\MI\|_{L^2(\Omega)} \frac{1}{\sqrt{\beta}}\| \nabla \bft\|_{L^3(\Omega)}
  \\
  \leq{}& \sqrt{{\mu}{}}
  \lVert \nabla \f v\rVert _{L^2(\Omega)} \sqrt{\beta}\| \bF-\MI\|_{L^2(\Omega)} \frac{C}{\sqrt{\beta\mu}}\| \nabla \bft\|_{L^3(\Omega)}
  \\
  \leq{}&\frac{\mu}{2}\| \nabla \f v\|_{L^2(\Omega)}^2+ \frac{C^2}{\beta \mu}\| \nabla \bft\|_{L^3(\Omega)}^2 \E(\f v,\bF) \,,
  \end{align*}
where we used that $ ( \MI \otimes \f v \dreidotkom \nabla \bft) = (\f v, \nabla \tr(\bft)) = 0 $. 
Furthermore, we obtain, by using the equations  $2 \| ( \nabla \f v )_{\skw}\|_{L^2(\Omega)}= \| \nabla \f v \|_{L^2(\Omega)}^2 = 2 \| ( \nabla \f v )_{\sym}\|_{L^2(\Omega)} $ for all $\f v \in \Hsig$, that
\begin{align*}
-2\left( ( \nabla \f v )_{\skw} (\bF-\MI); \bft\right)& - \alpha \left( ( \nabla \f v)_{\sym} (\bF-\MI); \bft\right) \\&\leq \left(2 \| ( \nabla \f v)_{\skw}\|_{L^2(\Omega)}+ \alpha \| ( \nabla \f v )_{\sym} \|_{L^2(\Omega)} \right)\| \bF-\MI\| _{L^2(\Omega)} \| \bft\|_{L^\infty(\Omega)} \\&\leq \frac{1}{2}( 2 +  \alpha ) \| \nabla \f v \|_{L^2(\Omega)} \| \bF-\MI\| _{L^2(\Omega)} \| \bft\|_{L^\infty(\Omega)} 
\\&\leq \frac{\mu}{2} \|\nabla \f v \|_{L^2(\Omega)} ^2 + \frac{(2+\alpha)^2}{
4\mu\beta}\| \bft\|_{L^\infty(\Omega)}^2 \E(\f v ,\bF) \,.
\end{align*}
Moreover, we observe that
\begin{align}
    -\delta \left(  (\bF - \MI)^2;\bft\right) \leq \frac{2\delta}{\beta}\| (\bft)_{-}\|_{L^\infty(\Omega)} \E(\f v,\bF).
\end{align}
Notice that we always manipulated the appearing terms of $\bF$ by subtracting the identity in order to have a good $L^2$ estimate in terms of the energy. The additional terms only give linear contributions and do not affect the analysis.

These three inequalities imply that the mapping 
\begin{align}\label{convexA1}
\begin{split}
        ( \f v ,\bF) \mapsto & \mu \| \nabla \f v \|_{L^2(\Omega)}^2 + \left( \f v \otimes \f v - \alpha\beta(\bF-\MI)^2 ; \nabla \vv\right) + \left( (\bF-\MI) \otimes \f v \dreidotkom \nabla \bft\right)\\ &+2\left( ( \nabla \f v )_{\skw} (\bF-\MI); \bft\right) + \alpha \left( ( \nabla \f v)_{\sym} (\bF-\MI); \bft\right)
    \\ &+\left( 2\max\{1,\alpha\}\| (\nabla \vv)_{\sym,-} \|_{L^\infty(\Omega ; \R^{d\times d})} + \frac{C^2}{\beta \mu}\| \nabla \bft\|_{L^3(\Omega)}^2\right)  \E(\f v ,\bF)
 \\ &+\left(  \frac{(2+\alpha)^2}{
4\mu\beta}\| \bft\|_{L^\infty(\Omega)}^2+\frac{2\delta}{\beta}\| (\bft)_{-}\|_{L^\infty(\Omega)} \right)  \E(\f v ,\bF)
\end{split}
\end{align}
is nonegative. 
Moreover, since it is quadratic, it is also a convex mapping.
 The linearity of the mapping
 \begin{align}\label{convexA2}
 \begin{split}
       ( \f v, \bF) \mapsto & - \mu \langle \nabla \f v , \nabla \vv \rangle + \alpha(  (1-2 \beta) (\bF-\MI)  ; \nabla \vv) + \alpha  \left( ( \nabla \f v)_{\sym} ; \bft\right)
  \\&- \langle \f f , \f v - \vv \rangle - ( (1-\delta)(\bF- \MI); \bft) 
 \end{split}
 \end{align}
 assures its convexity and therewith also its weakly-lower semi-continuity~\cite{ioffe}.

	For the remaining dissipative terms, we find
\begin{align*}
   & \left( (1-\beta)  ( \MI - \bF^{-1}) + (\beta +\delta(1-\beta)) (\bF-\MI) + \delta \beta \bF (\bF-\MI)\right): ( \bF-\MI)
\\
&=  \tr\left(  (1-\beta)  \bF^{-1}- ( 1- \beta (1-\delta))\MI + (1-\beta - 3 \delta \beta) B + (\beta + \delta - 3 \delta \beta) (B-\MI)^2 + \delta \beta \bF^3
\right)
\,.
\end{align*}
Defining the mapping    $\mathcal{F}: \R^{d\times d} _{\sym,+} \ra [0,\infty) $ via 
\begin{align}\label{convexA3}
\begin{split}
        \mathcal{F}(\bF)   : ={}& (1-\beta)\tr\left(    \bF^{-1}) \right) - ( 1- \beta (1-\delta))d + (1-\beta - 3 \delta \beta) \tr(\bF) + (\beta + \delta - 3 \delta \beta) |\bF-\MI|^2\\& + \delta \beta \tr(\bF^3)
 + (\beta + \delta - 3 \delta \beta)_- | \bF - \MI|^2, 
\end{split}
\end{align}
we observe that $\mathcal{F}$ is convex. 
 Indeed, the constant and linear terms are trivially convex. The condition $ \bF\in \R^{d\times d}_{\sym,+}$ assures that the mapping $ \bF \mapsto \tr(\delta \beta \bF^3  +(1-\beta) \bF^{-1})$  is convex. 
For the quadratic term, we observe that  by adding $ (\beta + \delta - 3 \delta \beta)_- \E( \f v , \bF)$, where $( x)_-$ denotes the negative part of a real number, this term become non	negative and thus convex.  
Combining~\eqref{convexA1},~\eqref{convexA2}, and~\eqref{convexA3}, we find that the condition~\eqref{ass:convex} is fulfilled such that Theorem~\ref{thm:envar} can be applied and assures the existence of a solution in the sense of Definition~\ref{def:visco2}.

Now, we want to prove the second part of Theorem~\ref{thm:visco2} concerning the weak-strong uniqueness of solutions. Therefore, we want to apply Proposition~\ref{prop:rel} and we have to assure that Hypothesis~\ref{hypo:2} is fulfilled and that a strong solution enjoying the regularity~\eqref{regStrongM} fulfills the regularity assumptions of~\eqref{abstractreg}. 

For this example we define $\W=\Hsig\times L^3_{\sym}$ and $p=2$ such that
\begin{align*}
\mathbb{Z} (0,T)&:= L^2(\otimeT;\Hsig\times L^3_{\sym})\cap L^\infty(\otimeT;\Ha\times L^2_{\sym}) \\
\intertext{and}
\widetilde{\mathbb Z}(0,T)&:= L^1(\otimeT;L^2(\Omega;\R^d)\times L^2_{\sym}) \cap L^2(\otimeT;(\Hsig)^*\times L^{3/2}_{\sym})\,.
\end{align*}
Firstly, we observe that  $\tbF\in L^\infty(\Omega\times \otimeT )$ and its positive definiteness $ \det \tbF \geq b >0$ a.e.~in $\Omega \times (0,T)$ implies that also $\bF^{-1} \in L^\infty(\Omega\times \otimeT )$. 
This implies, by the calculation~\eqref{subEM}, that for $( \tilde{\f v},\tilde{\bF})\in L^\infty(\otimeT;\Y)$ it holds that $\partial \E(\tilde{\f v},\tilde{\bF})\in L^\infty(\otimeT;\Y)$.

Secondly, we rewrite 
 the operator $A$, via
\begin{align*}
  \int_0^{T}  \langle A(t,\tilde{\f v}, \tilde{\bF}), (\tilde{\f w},\tilde{\mathbb G})\rangle \de t  ={} &
\int_0^T   \mu \langle \nabla \tilde{\f v}, \nabla \tilde{\f w}\rangle - \left( \tilde{\f v} \otimes \tilde{\f v} -\alpha ((1-\beta)(\tilde{\bF} - \MI )+ \beta  (\tilde{\bF}^2-\tilde{\bF})  )  ; \nabla \tilde{\f w}\right) \de t  \\
 & - \int_0^T \langle \f f , \tilde{\f w}\rangle 
  - \left( ( \tilde{\f v}\cdot\nabla) \tilde{\bF}  ; \tilde{\mathbb G}\right) + 2\left( ( \nabla \tilde{\f v} )_{\skw} \tilde{\bF}; \tilde{\mathbb G}\right)\de t \\& - \int_0^T \alpha \left( ( \nabla \tilde{\f v})_{\sym} \tilde{\bF}; \tilde{\mathbb G}\right) + \left( \tilde{\bF} -\MI + \delta  (\tilde{\bF}^2- \tilde{\bF}), \tilde{\mathbb G}\right) \de t \,
  \end{align*}
    such that we may estimate
  \begin{align*}
 \| A(\cdot ,\tilde{\f v}, \tilde{\bF})\|_{\widetilde{\mathbb Z}(0,T)}  \leq{} &   \mu \| \tilde{ \f v}\|_{L^2(\Hsig)}\| 
+ \| \tilde{\f v}\|_{L^4(L^4)}^2
+ \alpha ((1-\beta) 3+\beta \lVert \tilde{\bF}\rVert_{L^\infty(\Omega\times\otimeT)})  \lVert \tilde{\bF}-\MI\rVert_{L^2(L^2_\sym)} 
  \\
  &+ \| \f f \|_{L^2((\Hsig)^*}
  + \| \tilde{\f v}\|_{L^2(L^6)}\| \tilde{\bF}\|_{L^2(W^{1,3})}
  +(2+|\alpha|) \| \tilde{\f v}\|_{L^2(\Hsig)}\| \tilde{\bF}\|_{L^2(L^\infty)}
  \\
  &+ ( 3+\delta \lVert \tilde{\bF}\rVert_{L^2(L^\infty)})  \lVert \tilde{\bF}-\MI\rVert_{L^2(L^2_\sym)}
   \,.
\end{align*}
Hence we get that 
$A(t,\tilde{\f v},\tilde{\bF})\in \widetilde{\mathbb Z}\otimeT $. 
From the regularity assumption~\eqref{regStrongM}, we also infer $\t (\f v ,\bB)^T \in \widetilde{\mathbb Z}\otimeT$.

Thirdly, we observe that $ A $ and $\Psi$ are polynomials in $\nabla \f v$, $\f v$, $\bF$, and $\bF^{-1}$ and as such Gateaux-differentiable (c.f.~\eqref{visco:diss2}). 
All appearing terms are well-defined for the given regularities such that the conditions $D\Psi (\tilde{\f v},\tilde{\bF}) \in \widetilde{\mathbb Z}\otimeT$ and  $DA(\tU)\in  \mathcal L ( \mathbb{Z}(0,T) ; L^1(\otimeT;\Y^*))$ are fulfilled. 

Finally, we calculate the Hessian of the energy, 
\begin{equation}\label{secondE}
     D ^2 \E(\tilde{\f v},\tilde{\bF}) = \begin{pmatrix}
     \MI & 0\\0&(1-\beta)\bF^{-1}\otimes \bF^{-1}+\beta \MI\otimes \MI  \,
 \end{pmatrix}
\end{equation}
and need to prove that $D^2\E(\tilde{\f v},\tilde{\bF} )\in \mathcal L ( \mathbb Z\otimeT,\mathbb Z \otimeT) $. Since in the first entry $D^2\E(\tilde{\f v},\tilde{\bF} )$ is the identity, the property is obvious. Since $\bF^{-1} \in L^\infty(\Omega \times \otimeT)$ and $\mathbb Z\otimeT $ does not include any spatial derivative for the second component, this property also follows in the second variable of~$D^2\E(\tilde{\f v},\tilde{\bF} ) $. 
Hence, we proved that the conditions~\eqref{abstractreg} are fulfilled such that the weak-strong uniqueness result follows from Proposition~\ref{prop:rel} and Corollary~\ref{cor:weakstrong}. 


\end{proof}

\renewcommand{\bF}{\mathbb{F}}

\section{Oldroyd-B model for viscoelastic fluid via symmetrized \\Neo--Hook\-ean approach} \label{sec:per}
%

 In this section, we want to apply the abstract result of Theorem~\ref{thm:envar}  to the system~\eqref{fsim2}, which we recall here for the readers' convenience
 \begin{subequations}
\label{sysvisco}     
 \begin{align}
\t \f v + ( \f v \cdot \nabla ) \f v + \nabla p - \mu \Delta \f v  - \alpha  \di \left ( \bF^2- \MI  \right ) ={}& \f f \,, \quad \di \f v = 0 \,,\label{eq1sym}\\
\t \bF +  ( \f v \cdot \nabla ) \bF -2[ (\nabla \f v)_{\skw} \bF]_{\sym} - \alpha [  ( \nabla \f v )_{\sym} \bF]_{\sym}  + \frac{1}{\mu_p}  ( \bF - \bF^{-1}) ={}& 0 \,, \quad  ( \bF)_{\skw} = 0 \, .\label{eq2sym}
\end{align}
We take the parameters $\alpha \in \R\backslash \{0\}, \mu ,\,\mu_p\in[0,\infty)$ and we  equip the system  with the following initial and boundary conditions, 
\begin{align}
    \f v(0) = \f v_0 \,, \quad \bF(0) = \bF_0 \quad \text{in } \Omega\\
    \f v = 0 \quad \text{on } \partial\Omega  \times \cltime\,.
\end{align}
 \end{subequations}

We define the space $ \V $ to be given  by $\V := \Ha(\Omega)\times L^2_{\sym}(\Omega)$ and the energy $\E : \V \ra [0,\infty)$ via
\begin{align}\label{visco:energy}
    \E(\f v , \bF) : = 
    \begin{cases}
        \frac{1}{2} \left[ \lVert \f v\rVert^ 2_{L^2(\Omega)} +\int_\Omega \lvert \bF\rvert^2 - | \MI|^2 - \log (\det (\bF^2))     \de\f x\right] & \text{if } \bF \in L^2_{\sym,+} (\Omega)
        \\
        \infty & \text{else}
    \end{cases}
    \,.
\end{align}
The space $\Y$ can be defined as in~\eqref{spaceY}.

\begin{definition}\label{def:visco}
We call the triple  $$(\f v, \bF, E)\in L^\infty(\otime;\Ha(\Omega))\cap L^2(\otime;\Hsig(\Omega)) \times L^\infty(\otime;L^2_{\sym,+}(\Omega)) \times \BVl$$ an \textit{energy-variational solution} to \eqref{eq1sym}--\eqref{eq2sym} with the initial values $\f v_0\in \Ha(\Omega)$ and $\bF_0\in L^2_{\sym,+}(\Omega)$, if $ \mathcal{E}(\f v, \bF) \leq E $ and
\begin{align*}
&\left[E - ( \f v , \vv) - ( \bF , \bft )  \right]\Big|_s^t &\\&+ \int_s^t ( \f v , \t \vv ) + ( \bF , \t \bft) + \mu ( \nabla \f v , \nabla \f v - \nabla \vv) + ( \f v\otimes \f v -\alpha  \bF^2 - \MI  : \nabla \vv)  - \langle \f f , \f v - \vv \rangle \de \tau &\\
&+ \int_s^t   (( \f v \cdot\nabla  )\bft , \bF) +2 (  (\nabla \f v)_{\skw} \bF ,  \bft) +\alpha ( (\nabla \f v)_{\sym} \bF , \bft) + \frac  {1 }{\mu_p }\left( \bF - \bF^{-1} , \bF - \bF^{-1} - \bft \right) \de \tau &
\\
&+ \int_s^t \mathcal{K}( \vv , \bft )  \left( \mathcal{E}(\f v , \bF) - E \right) \de \tau &\leq 0 \,
\end{align*}
holds for all $(\vv,\bft) \in \C^1(\cltime ;\Y)$ and for a.e.~$\infty>t>s\geq 0 $ including $s=0$ with $ \f v(0) = \f v_0 $ and $ \bF (0) = \bF_0$,
where the regularity weight $\mathcal{K}$ is given by
\begin{align*}
\mathcal{K}(\vv , \bft) ={}& 2\max\{1,\alpha\} \| ( \nabla \vv)_{\sym}\|_{L^\infty(\Omega)} +   \frac{(2+\alpha)^2}{4\mu} \| \bft\|_{L^\infty(\Omega)} ^2\\
&+ \frac{1}{\mu_p}\| (\bft)_{-}\|_{L^\infty(\Omega)}^2+ \frac{C^2}{\mu} \| \nabla \bft\|_{L^3(\Omega)}^2 \,.
\end{align*}
\end{definition}

Then, we get the following. 


\begin{theorem}\label{thm:visco}
For every $ \f v _0 \in \Ha(\Omega)$, $ \bF_0 \in L^2_{\sym,+}(\Omega)$ with $-\ln \det \bF_0  \in L^1(\Omega)$ and right-hand side $ \f f \in L^2(\otime ; (\Hsig(\Omega))^*)$ there exists an \textit{energy-variational solution} in the sense of Definition~\ref{def:visco} with $ \E(\f v_0, \bF_0) = E(0)$.

Let $ (\tv , \tbF)$ be a weak solution to~\eqref{sysvisco} with 
\begin{align}\label{regStrongF}
\begin{split}
        \tv & \in L^\infty(\otimeT;W^{1,\infty}(\Omega)) 
        \cap C^1(\cltimeT; \Ha(\Omega))\\
    \tbF &\in L^\infty (\otimeT;(W^{1,3}\cap L^\infty)( \Omega; \R^{d\times d}_{\sym,+})) \cap C^1(\cltimeT; L^2_{\sym}(\Omega))
\end{split}
\end{align}
such that there exists a $b\in(0,1)$ with $ \det \tbF \geq b$ for a.e.~$(x,t)\in\Omega \times (0,T)$
as well as $ (\tv(0),\tbF(0))=(\f v_0,\bF_0) $. Then it holds that 
$$
\f v = \tv \quad\text{and}\quad \bF = \tbF \quad \text{ for all }(x,t)\in\Omega \times \otimeT\,.
$$

\end{theorem}
%
%
\begin{proof}
In order to prove Theorem~\ref{thm:visco}, we want to apply Theorem~\ref{thm:envar}. Therefore, we have to show that Hypothesis~\ref{hypo:1} is fulfilled. 
As above, letting $\V := \Ha(\Omega)\times L^2_{\sym}(\Omega)$ and the energy be given by~\eqref{visco:energy}.
The dissipation potential $\Psi : \V \ra [0,\infty]$ is given by
\begin{align}
    \label{visco:diss}
    \diss({t},\f v , \bF) = \begin{cases}
        \int_\Omega \mu\lvert \nabla \f v \rvert^2 - \langle \f f {(t)}, \f v \rangle + \frac{1}{\mu_p} \lvert \bF-\bF^{-1} \rvert^2 \de x  & \text{if }\f v \in \Hsig, \, \bF \in L^2_{\sym,+}(\Omega) \\
        \infty & \text{else}
    \end{cases}\,
\end{align}
on its domain.
In a standard way, we may choose $ C_{\Psi}(t) := \frac{1}{2\mu}  \| \f f \|_{(\Hsig)^*}^2$ and estimate via Young's inequality that $\Psi (t,v,\bF) \geq -C_{\Psi}(t) $ as in the previous example.
In order to check the next assumptions, we need to calculate the convex conjugate of the energy $\E$. 
The subdifferential is single-valued on its domain and is given by  \begin{align}\label{subEF}
\partial \E : \V \ra 2^{\V^*} \qquad
    \partial \E( \f v , \bF) =\begin{pmatrix}
        \f v \\ \bF - \bF^{-1} 
    \end{pmatrix}\,.
\end{align}
Its inverse gives the subdifferential of the convex conjugate according to Fenchels equivalences, which is  single-valued on $\V^*= \V $ and given by 
\begin{align}\label{subEstar}
    \partial \E^* : \V^* \ra 2^{\V} \qquad
    \partial \E^*(\vv,\bft) = 
        \begin{pmatrix}
        \vv\\
        \f{F}(\bft)
    \end{pmatrix} \, \hbox{when $ \bft \in L^2_{\sym,+}$ and $(\vv , \f 0)^T$ else.}
\end{align}
Here,  we used the definition $\f{F}(\bft) = 
 \frac{\bft}{2}+\sqrt{\frac{\bft^2}{4}+\MI}$,
  which comes from the fact that  $\bft$
solves the matrix equation $\bft = \bF- \bF^{-1} $ and as a consequence of the positive definitesness of $\bF$, it holds  $ 0=\bF^2 - \bft \bF - \MI $.
  The plus sign in front of the square root in the definition of $\f{F}$ reflects the fact that $\partial \E^*$ should map into the domain of $\E$. 
We note that the derivative $D\E^*(\vv,\bft) \in \dom\diss$ for all $ (\vv,\bft)\in\Y$, which follows from the fact that $\vv \in \Hsig$ and $ \bft \in L^2_{\sym,+} $ implying that the second component of $ \partial \E^*$ belongs to $ L^2_{\sym,+}(\Omega)$.   

The operator $A: \dom\diss \ra \Y^*$ is given by all terms in the equations~\eqref{sysvisco} despite the time derivative, such that
\begin{align}\label{defA}
\begin{split}
    \langle A(\f v, \bF), (\vv,\bft)^T\rangle = &
  \mu \langle \nabla \f v, \nabla \vv\rangle - \left( \f v \otimes \f v - \alpha\bF^2 ; \nabla \vv\right)  - \langle \f f , \vv\rangle 
  \\
 & - \left( \bF \otimes \f v \dreidotkom \nabla \bft\right) - 2\left( ( \nabla \f v )_{\skw} \bF; \bft\right) - \alpha \left( ( \nabla \f v)_{\sym} \bF; \bft\right) + \frac{1}{\mu_p} \left( \bF - \bF^{-1} , \bft\right) \,.
 \end{split}
\end{align}
We need to verify that~\eqref{Adiss} is fulfilled for $A$ and $\diss$ as given above, \textit{i.e.,}
$$
\left\langle A( {t,}D \E^*(\vv , \bft)) ; \begin{pmatrix}
\vv\\\bft
\end{pmatrix}\right \rangle = \Psi(t,D \E^*(\vv,\bft))
$$
In order to verify this assumption, we insert $\f v= \vv$ and $\bF = \f{F}(\bft)$ in~\eqref{defA}. 
This calculation resembles the usual energy estimate. 
 Since, $\vv$ is a solenoidal vector field, the convection term vanishes, $ ( \vv \otimes \vv ;\nabla \vv)=0$. From the identity $ \f{F}^2(\bft) - \bft \f{F}(\bft) =I$, we find that 
 $$
 \alpha ( \f{F}^2(\bft) ;\nabla \vv) - \alpha (( \nabla \vv )_{\sym} \f{F}(\bft); \bft) = \alpha ( \f{F}^2(\bft) - \bft \f{F}(\bft) ; (\nabla \vv)_{\sym}) = \alpha (\MI ; \nabla \vv) = 0 \,.
 $$
 Furthermore, we observe by an integration by parts and expressing $\bft$ in terms of $\f{F}(\bft)$ that
 \begin{align*}
 -& ( \f{F}(\bft)\otimes \vv \dreidotkom \nabla \bft) - ( ( \nabla \vv)_{\skw}\f{F}(\bft);\bft) 
 \\&= ( ( \vv \cdot \nabla ) \f{F}(\bft) ;\mathbb{F}(\bft) - (\f{F}(\bft) )^{-1}) - ( ( \nabla \vv)_{\skw}\f{F}(\bft) ; \mathbb{F}(\bft) - (\f{F}(\bft) )^{-1}) 
 \\&= \int_\Omega ( \vv \cdot \nabla ) \tr\left ( \frac{1}{2}\f{F}^2(\bft) - \ln (\f{F}  ( \bft))\right ) - ( \nabla \vv)_{\skw} : (\f{F}^2(\bft) - \MI ) \de x = 0\,.
 \end{align*}
 The first term vanishes since $\vv$ is a solenoidal vector field and the second one since $\f{F}^2(\bft) - \MI$ is symmetric. 
Using all these cancellations, we find  for the remaing dissipative terms the following equality
\begin{align*}
  \langle A({t,}\vv, \f{F}(\bft)), (\vv,\bft)^T\rangle  ={} & \mu \| \nabla \vv\|_{L^2(\Omega)}^2 - \langle \f f{(t)} , \vv\rangle + \frac{1}{\mu_p} \| \bft \|_{L^2(\Omega)}^2 = \diss ( \vv , \f{F}(\bft))\,
\end{align*} 
 holding true  for all $ (\vv,\bft)\in \Y$. 
 Due to~\eqref{subEstar}, this calculation implies that assumption~\eqref{Adiss} is fulfilled and 
also that the mapping $ \diss \circ D\E^*(\vv, \bft) = \diss( \vv, \mathbb{F}(\bft)) $ given above is continuous on $\Y$. 
%
%
{We also observe from \eqref{visco:energy} and \eqref{visco:diss} that at their global minimum $\E(\f v_{\min} , \bF_{\min})=\diss(\f v_{\min} , \bF_{\min})=0$ with $ \f v_{\min}= \f 0$ and $\bF_{\min} = \MI$.} 

The final condition, we need to verify is~\eqref{ass:convex}. 
Therefore, we observe the estimates
\begin{align*}
    \left( \f v \otimes \f v - \alpha\bF^2 ; \nabla \vv\right)={}&
    \left( \f v \otimes \f v - \alpha\bF^2 ; (\nabla \vv)_{\sym}\right)
\\
    \leq{}& \left(\| \f v \|_{L^2(\Omega)}^2 + \alpha \| \bF\|_{L^2(\Omega)}^2 \right) \| (\nabla \vv)_{\sym} \|_{L^\infty(\Omega ; \R^{d\times d})} 
 \\   \leq{}&  \| (\nabla \vv)_{\sym} \|_{L^\infty(\Omega ; \R^{d\times d})} 2\max\{1,\alpha\}\E(\f v ,\bF)
 \intertext{and}
  \left( \bF \otimes \f v \dreidotkom \nabla \bft\right) &\leq{} \lVert \f v\rVert _{L^6(\Omega)} \| \bF\|_{L^2(\Omega)} \| \nabla \bft\|_{L^3(\Omega)}
 \\& \leq C \| \nabla \f v\|_{L^2(\Omega)} \| \bF\|_{L^2(\Omega)} \| \nabla \bft\|_{L^3(\Omega)}
 \\ & \leq  \frac{\mu}{2} \| \nabla \f v\|_{L^2(\Omega)}^2 + \frac{C^2}{\mu} \| \nabla \bft\|_{L^3(\Omega)}^2 \E(\f v ,\bF) \,.
\end{align*}
Furthermore, we obtain,  from the fact that $2 \| ( \nabla \f v )_{\skw}\|_{L^2(\Omega)}= \| \nabla \f v \|_{L^2(\Omega)}^2 = 2 \| ( \nabla \f v )_{\sym}\|_{L^2(\Omega)} $
for all $\f v \in \Hsig$, that
\begin{align*}
-2\left( ( \nabla \f v )_{\skw} \bF; \bft\right)& - \alpha \left( ( \nabla \f v)_{\sym} \bF; \bft\right) \\&\leq \left(2 \| ( \nabla \f v)_{\skw}\|_{L^2(\Omega)}+ \alpha \| ( \nabla \f v )_{\sym} \|_{L^2(\Omega)} \right)\| \bF\| _{L^2(\Omega)} \| \bft\|_{L^\infty(\Omega)} \\&\leq \frac{1}{2}( 2 +  \alpha ) \| \nabla \f v \|_{L^2(\Omega)} \| \bF\| _{L^2(\Omega)} \| \bft\|_{L^\infty(\Omega)} 
\\&\leq \frac{\mu}{2} \|\nabla \f v \|_{L^2(\Omega)} ^2 + \frac{(2+\alpha)^2}{
4\mu}\| \bft\|_{L^\infty(\Omega)}^2 \E(\f v ,\bF) \,.
\end{align*}
These three inequalities imply that the mapping 
\begin{align*}
    ( \f v ,\bF) \mapsto & \mu \| \nabla \f v \|_{L^2(\Omega)}^2 + \left( \f v \otimes \f v - \alpha\bF^2 ; \nabla \vv\right) + \left( \bF \otimes \f v \dreidotkom \nabla \bft\right)+\left( ( \nabla \f v )_{\skw} \bF; \bft\right) + \alpha \left( ( \nabla \f v)_{\sym} \bF; \bft\right)
    \\ &+\left( 2\max\{1,\alpha\}\| (\nabla \vv)_{\sym} \|_{L^\infty(\Omega ; \R^{d\times d})} +\frac{C^2}{\mu} \| \nabla \bft\|_{L^3(\Omega)}^2 +  \frac{(2+\alpha)^2}{4\mu}\| \bft\|_{L^\infty(\Omega)} ^2\right)  \E(\f v ,\bF)
\end{align*}
is nonegative. Since it is quadratic, it is also a convex mapping. The linearity of the mapping, $$  ( \f v, \bF)\mapsto - \mu \langle \nabla \f v , \nabla \vv \rangle - \langle \f f , \f v - \vv \rangle - \frac{1}{\mu_p}( \bF, \bft) $$ assures its convexity and therewith, weakly-lower semi-continuity~\cite{ioffe}.  

For the dissipative term depending on the stress tensor $\bF$, we find
\begin{align}
    \lVert \bF-\bF^{-1}\rVert_{L^2(\Omega)}^2 =  \lVert \bF\rVert_{L^2(\Omega)}^2+ \lVert \bF^{-1}\rVert_{L^2(\Omega)}^2 - 2 d .
\end{align}
The first term is obviously quadratic and convex and the constant term, does not change this. 
Finally, we consider the mapping $\mathcal{S}:\R^{d\times d} _{\sym,+} \ra \R$ given via
$$
\mathcal S : \bF \mapsto \tr( (\bF^{-1}) ^2) +  \tr(\bF^{-1} \Phi) - \tr((\Phi)_-)^2 \ln \det(\bF) \,.
$$
The first derivative in direction $\mathbb{G}\in \R^{d\times d}_{\sym,+}$ is given by
$$\langle D\mathcal{S}(\bF) , \mathbb G\rangle = - \tr( \bF^{-1} \bF^{-1} \mathbb G \bF^{-1} ) - \tr( \bF^{-1} \mathbb G \bF^{-1} \bF^{-1} ) - \tr(\bF^{-1} \mathbb G \bF^{-1}  \Phi) - \tr((\Phi)_-)^2 \tr(\bF^{-1}\mathbb G) \,.
$$
Calculating the second derivative, in the directions $\mathbb{G}$ and $\mathbb{H}\in \R^{d\times d}_{\sym,+}$, we find
\begin{align*}
    \langle D^2 \mathcal S ( \bF)) ,( \mathbb{G},\mathbb{H})\rangle& = 
\tr( \bF^{-1}\mathbb{H} \bF^{-1} \bF^{-1} \mathbb{G}\bF^{-1} + \bF^{-1} \bF^{-1} \mathbb{H}\bF^{-1} \mathbb{G}\bF^{-1} +\bF^{-1} \bF^{-1} \mathbb{G}\bF^{-1} \mathbb{H}\bF^{-1})\\&\quad  +\tr( \bF^{-1} \mathbb{H}\bF^{-1} \mathbb{G}\bF^{-1} \bF^{-1} +\bF^{-1} \mathbb{G}\bF^{-1} \mathbb{H}\bF^{-1} \bF^{-1} +\bF^{-1} \mathbb{G}\bF^{-1} \bF^{-1} \mathbb{H}\bF^{-1} )
\\&\quad +
\tr ( \Phi (\bF^{-1} \mathbb{G} \bF^{-1} \mathbb{H} \bF^{-1} + \bF^{-1} \mathbb{H} \bF^{-1} \mathbb{G} \bF^{-1} )) \\&\quad+ \tr((\Phi)_-)^2 \tr( \mathbb{G} \bF^{-1} \mathbb{H} \bF^{-1} +  \mathbb{H} \bF^{-1} \mathbb{G} \bF^{-1})
\intertext{such that}
 \langle D^2 \mathcal S ( \bF)) ,( \mathbb{G},\mathbb{G})\rangle
 &= 
2\tr( (\bF^{-1}\mathbb{G} \bF^{-1})^2 
+ \bF^{-1} \bF^{-1} (\mathbb{G}\bF^{-1} )^2 + \bF^{-1} (\mathbb{G}\bF^{-1} )^2\bF^{-1} 
)
\\&\quad +
2\tr ( \Phi (\bF^{-1} (\mathbb{G} \bF^{-1})^2
))
+2 \tr((\Phi)_-)^2 \tr( (\mathbb{G} \bF^{-1})^2)
\\
 &\geq 
2\tr( (\bF^{-1}\mathbb{G} \bF^{-1})^2 
+ \bF^{-1} \bF^{-1} (\mathbb{G}\bF^{-1} )^2 + \bF^{-1} (\mathbb{G}\bF^{-1} )^2\bF^{-1} 
)
\\&\quad - 2
\tr ( (\Phi)_- ) \tr(  \bF^{-1}( \mathbb{G} \bF^{-1})^2 
)+2 \tr((\Phi)_-)^2 \tr( (\mathbb{G} \bF^{-1})^2 )
\\&\geq 2\tr( (\bF^{-1}\mathbb{G} \bF^{-1})^2 
+ \bF^{-1} \bF^{-1} (\mathbb{G}\bF^{-1} )^2 + \bF^{-1} (\mathbb{G}\bF^{-1} )^2\bF^{-1} 
)>
0\,
\end{align*}
as long as $ \mathbb G \neq \MO$. 
The first inequality follows from
\begin{align}
    \tr (   (\Phi)_+ \bF^{-1} \mathbb{G} \bF^{-1}\mathbb{G} \bF^{-1} ) =  \tr (  \bF^{-1/2} (\Phi)_+\bF^{-1/2}( \bF^{-1/2} \mathbb{G} \bF^{-1/2})^2)
    \geq 0\,.
\end{align}
since the positive semi definite matrix $\bF^{-1/2} (\Phi)_+\bF^{-1/2}$ induces a positive quadratic form. 
Note that $\bF$ is a positive definite matrix such that also their inverse and products remain positive definite.
In order to deduce the second inequality, we used the trace property
$ \tr(\mathbb A \mathbb B) \leq \tr( \mathbb A^2 )\tr(\mathbb B^2)$, which implies 
$$ \tr(\mathbb A \lambda \mathbb B^2) = \tr (\mathbb  A {\mathbb B}\lambda {\mathbb B} ) \leq \tr((\mathbb A{\mathbb B})^2)+ \tr((\lambda{\mathbb B} )^2)\leq \tr (\mathbb A^2 \mathbb B^2) + \lambda^2\tr(\mathbb B^2)$$ for positive definite symmetric matrices $\mathbb A$, and $\mathbb B$, as well as $\lambda >0$. 
The inequality 
\begin{align*}
\tr ( (\Phi)_- ) \tr(  \bF^{-1}( \mathbb{G} \bF^{-1})^2 )
\leq \tr(\bF^{-1} \bF^{-1}( \mathbb{G}\bF^{-1})^2  ) + \tr((\Phi)_-)^2 \tr(( \mathbb{G} \bF^{-1})^2  )    
\end{align*}
follows from by choosing $ \mathbb A=  \bF^{-1}$, $\mathbb B =  \mathbb{G}\bF^{-1} $ and $\lambda = \tr ( (\Phi)_- )$.

Thus, we have shown that the mapping $\mathcal{S}$ is convex, since its second derivative is positive definite. 
Hence~\eqref{ass:convex} is fulfilled and thus, Hypothesis~\ref{hypo:1} is fulfilled, which implies the first  assertion concerning the existence.

Now we want to argue that Hypothesis~\ref{hypo:2} is fulfilled as well as the regularity assumption~\eqref{abstractreg} such that we may apply Proposition~\ref{prop:rel} and Corollary~\ref{cor:weakstrong}.
In this regard, we choose $\W :=\Hsig(\Omega) \times L^2_{\sym}(\Omega)  $, $p=2$  such that 
\begin{align*}
\mathbb{Z} (0,T)&:= L^2(\otimeT;(\Hsig\times L^2_{\sym})(\Omega) )\cap L^\infty(\otimeT;(\Ha \times L^2_{\sym})(\Omega))
\intertext{and}
\widetilde{\mathbb Z}(0,T)&:= L^1(\otimeT;(L^2\times L^2_{\sym})(\Omega)) \cap L^2(\otimeT;((\Hsig)^*\times L^{2}_{\sym})(\Omega))\,.
\end{align*}
{Similarly} to the previous example, we observe that $\tilde{\bF}^{-1} \in L^\infty(\Omega\times \otime )$ because  $ \det\tilde{ \bF} \geq b >0$ a.e.~in $\Omega \times (0,T)$. 
This implies, by the calculation~\eqref{subEF}, that, for $( \tilde{\f v},\tilde{\bF})\in L^\infty(\cltime;\Y)$, it holds that $D \E(\tilde{\f v},\tilde{\bF})\in L^\infty(\otimeT;\Y)$.
The validation of the regularity assumption~\eqref{abstractreg} for strong solutions fulfilling~\eqref{regStrongF} is very similar to the reasoning at the end of Section~\ref{sec:Mal}, hence we will not report this here in detail. Thus the second assertion of Theorem~\ref{thm:visco} follows from Propositon~\ref{prop:rel} and Corollary~\ref{cor:weakstrong}.

\end{proof}
\section{Conclusion and Outlook}
In the previous sections, we introduced and exemplified a new formulation and solvability concept for nonlinear evolution equations especially suited to treat viscoelastic fluid models. This formulation allows to pass to the limit in possible approximations only be means of compactness in weak topologies and thus, does not require to add  stress diffusion in viscoelastic fluid models. 
We are able to provide existence and weak-strong uniqueness under {proper} assumptions, which we show to be reasonable at the hand of two different examples in the Sections~\ref{sec:Mal} and~\ref{sec:per}. 

In order to further support our claim that this new approach has the potential to provide an  adaptable and transferable concept, we want to comment on the applicability to other viscoelastic fluid models and possible directions of future research in this last section. 

\paragraph{Peterlin model.}
Another viscoelastic fluid model that received a lot of attention in recent years with a quadratic (hence superlinear) elastic free energy density  
\begin{equation}
    \label{model29}
    \tilde{e}(\bB)=\frac{1}{2}\left (\tr (\bB)^2 -\tr(\MI)^2 - 2\tr (\ln (\bB)) \right)
\end{equation}
is the so-called viscoelastic Peterlin model for polymeric fluids introduced in~\cite{RenardyPeterlin} (\textit{cf.}~\cite{brunk}). Collecting \eqref{model2}, \eqref{model21}, \eqref{model23} and \eqref{model29}, and assuming, as in \cite{brunk}, that the parameter $\mu_p^{-1}$ is a linear function of $\tr \bB$ representing a generalized relaxation term as a material viscometric function,  we obtain the following model:
 \begin{subequations}
\label{pet1}     
 \begin{align}
\t \f v + ( \f v \cdot \nabla ) \f v + \nabla p - \mu \Delta \f v  -{\alpha}
\di \left ( (\tr \bB)\bB  \right ) ={}& \f f \,, \quad \di \f v = 0 \,,\label{pet1v}\\
\t \bB +  ( \f v \cdot \nabla ) \bB -2[ (\nabla \f v)_{\skw} \bB]_{\sym} - \alpha   [  ( \nabla \f v )_{\sym} \bB]_{\sym}  + \tr \bB\left((\tr \bB)\bB - \MI\right) ={}& \f 0 \,, \quad  ( \bB)_{\skw} = 0 \, .\label{pet1B}
\end{align}
\end{subequations}
The dissipation term \eqref{model23d} coming from the relaxation processes is in this case
\begin{align*}
    \Psi( \f v ,  \bB):= {}&
\int_{\Omega}\mu | \nabla \f v|^2 + (\tr \bB) \bB\left((\tr \bB)\MI-\bB^{-1}\right):\left((\tr \bB)\MI-\bB^{-1}\right)\de x  \\
={}&  \int_{\Omega}\mu | \nabla \f v|^2 + \tr(\bB)^4 - 2 \tr(\bB)^2+  \tr(\bB) \tr ( \bB^{-1})   \de x
\,.
\end{align*}
The last term in this dissipation potential is not convex as a mapping of $\bB$. 
Defining $ \mathcal{S}(\bB) = \tr (\bB)\tr( \bB^{-1}) $, we may calculate the second derivative as
\begin{align*}
    \langle \partial^2_{\bB} \mathcal{S} (\bB) , (\mathbb G,\mathbb H)\rangle &= - \tr(\mathbb G) \tr(\bB^{-1}\mathbb H\bB^{-1}) - \tr(\mathbb H) \tr(\bB^{-1}\mathbb G\bB^{-1}) \\&\quad + 2 \tr(\bB) \tr(\bB^{-1}\mathbb H\bB^{-1}\mathbb G\bB^{-1})\,.
\end{align*}
This form is not positive for $\mathbb G = \mathbb H$ and can not be made positive by adding a multiple of the energy. Indeed, we may estimate the two terms with the negative sign via
\begin{align*}
\langle \partial^2_{\bB} \mathcal{S} (\bB) , (\mathbb G,\mathbb G )\rangle&=
-2\tr(\mathbb G) \tr(\bB^{-1}\mathbb G\bB^{-1})  + 2 \tr(\bB) \tr(\bB^{-1}\mathbb G\bB^{-1}\mathbb G\bB^{-1}) \\&\geq - 2 \tr(\mathbb G)^2\frac{\tr(\bB^{-1})}{\tr(\bB)} \geq -2d \tr(\mathbb G)^2\tr( (\bB^{-1})^2)\,.
\end{align*}
However, from the second derivative of the energy, we get the term $ -\langle \partial^2_{\bB} \ln\det(\bB), (\mathbb G,\mathbb G)\rangle = \tr( (\bB^{-1} \mathbb G)^2)$. 
In order to have any hope to find a multiple of the energy that when added to $\mathcal{S}$ makes the sum convex, we would need to find a $\eta > 0$ $\tr(\mathbb G)^2\tr( \bB^{-1})^2\leq \eta \tr( (\bB^{-1} \mathbb G)^2)$. But this does not seem to be possible. Mainly, the problem lies in the non-convexity of the dissipation potential. 
Thus,  the hypothesis of convexity of the mapping \eqref{ass:convex} cannot be shown to be satisfied, and hence we cannot apply Theorem \ref{thm:envar} to infer the existence of \textit{energy-variational solutions} to system \eqref{pet1}. 

Collecting \eqref{model2}, \eqref{model21}, \eqref{model25} and \eqref{model29},  we obtain the following variant of the generalized Peterlin model:
 \begin{subequations}
\label{petg1}     
 \begin{align}
\t \f v + ( \f v \cdot \nabla ) \f v + \nabla p - \mu \Delta \f v  -{\alpha} 
\di \left ( (\tr \bB)\bB  \right ) ={}& \f f \,, \quad \di \f v = 0 \,,\label{petg1v}\\
\t \bB +  ( \f v \cdot \nabla ) \bB -2[ (\nabla \f v)_{\skw} \bB]_{\sym} - \alpha   [  ( \nabla \f v )_{\sym} \bB]_{\sym}  + \frac{1}{\mu_p}\left((\tr \bB)\bB^2 - \bB\right) ={}& \MO \,, \quad  ( \bB)_{\skw} = \MO \, .\label{petg1B}
\end{align}
\end{subequations}
The dissipation term \eqref{model25d} coming from the relaxation processes is in this case
\[
\Psi(\f v, \bB) := \int_{\Omega}\mu | \nabla \f v |^2 + \frac{1}{\mu_p} \left((\tr \bB)\bB-\MI\right):\left((\tr \bB)\bB-\MI\right)\de x.
\]
With the latter term in the system dissipation the hypothesis of convexity of the mapping \eqref{ass:convex} is satisfied. 
We can also allow more general dissipation functionals as long as the convexity assumption~\eqref{ass:convex} for the associated mapping is fulfilled. 

\paragraph{{Lin--Liu--Zhang model.}}
There are also other viscoelastic fluid models that fit into the proposed model framework. A very simple one arises, if we combine~\eqref{model6} with the energy $\tilde e (\bF) = \frac{1}{2}| \bF|^2$, this leads to the system considered in~\cite{Lin05},
\begin{align*}
    \t \f v + ( \f v \cdot \nabla ) \f v + \nabla p - \mu \Delta \f v  - \di \left ( \bF\bF^T  \right ) ={}& \f f \,, \quad \di \f v = 0 \,,\\
\t \bF +  ( \f v \cdot \nabla ) \bF -\nabla \f v\bF ={}& 0\,, \quad \di \bF = 0  \, .
\end{align*}
For this system all assumptions of abstract result in Theorem~\ref{thm:envar}
are fulfilled with regularity weight $\mathcal{K}_1(\vv,\bft)= 2 \| \nabla \vv \|_{L^\infty(\Omega)}+ C (\| \nabla \bft \|_{L^3(\Omega)}^2+ \| \bft\|_{{L^\infty(\Omega)}}^2)$, where the test functions are chosen as in the two examples above. 
In this context, the normalization condition~$\di\bF=\f 0$ could also be disregarded. 
But, by the div-curl lemma it can be observed that the weakly-lower semi-continuity of the associated functional in the assumption~\eqref{ass:convex} can already deduced for 
$\mathcal{K}_2(\vv,\bft)= 2 \| \nabla \vv \|_{L^\infty(\Omega)}$. Thus a future direction to further sharpen our result is to weaken the convexity assumptions of the associated form in~\eqref{ass:convex} in order to also sharpen the associated result for this basic example. 
Note that for the improved regularity weight $\mathcal K_2$, it can be shown that the existence of \textit{energy-variational solutions} implies the existence of measure-valued solutions as proposed in~\cite{Lin05} in the same way as it is proven for the Ericksen--Leslie system in~\cite{max}.

Moreover, it would be desirable to also include energies into our framework, that {do} not necessarily {have} superlinear growth, which could allow to further understand at least the Giese\-kus model~\eqref{gk1} or variants of it. We note that the superlinearity of the energy is essential to infer that $ \dom \E^* = \V^*$. In case that $\E$ is only of linear growth, its convex conjugate $\E^*$ may only be finite on a subset of $\V^*$. This especially hinders an argument as in \textit{Step 3.} of the proof of Theorem~\ref{thm:envar}, where the continuity of the function $f(\alpha)=\E(D\E^*(\alpha\Phi))$ is shown for all $\Phi \in \Y\subset \V^*$. 

But still it may be possible to show the continuity of the function on a smaller  (non-linear) set, $\Y \subset \dom \E^*$. This would lead to a set of test functions $\Y$ that is not linear anymore. But this can still be a reasonable class of solutions fulfilling existence and weak-strong uniqueness. Indeed, in the proof of the weak-strong uniqueness result in Proposition~\ref{prop:rel}, we have to chose $ \Phi = D \E(\tU)$ for the strong solution $\tU$. By the Fenchel equivalences~\eqref{eq:fenchel}
it is clear that $ D \E(\f y)\subset \dom D\E^* \subset \dom \E^* $ such that it would make sense to restrict the test space to $\Y\subset \dom \E^*$. We plan to investigate such models in the future.

Thus, we think that this new approach may provides some fruitfull new insight into the analysis of viscoelastic fluid models. 
The proposed formulation includes a large class of models and incorporates desirable properties. Beside the existence and weak-strong uniqueness, we observed that the solutions form a semi flow (see Remark~\ref{rem:semi}). It can also be inferred that the set of solutions is weakly-star closed. Therefore, it could be desirable to select a special solution of the  set of solutions as the physically relevant one as proposed in~\cite{maxdiss} for incompressible fluid dynamics.  We note that this is already the idea of the time-incremental minimization algorithm proposed in the proof of Theorem~\ref{thm:envar}, where the energy is minimized in~\eqref{eq:timedis}. 

\appendix
\section*{Appendix}

This appendix has two parts. Firstly, we give an explicit formula for the matrix $\mathbb W$ in~\eqref{NoeHookOldroyd} and Secondly, we calculate the abstract relative energy-inequality~\eqref{relenin} for the first example~\eqref{sysvisco2} explicitly.
\subsection*{{A: explicit formula for the matrix $\mathbb W$}}
Let us consider the polar decomposition of the deformation gradient, i.e.
\[
\mathbb{F}= \mathbb{S} \mathbb{Q},
\]
valid if $\det \mathbb{F}>0$, where $\mathbb{Q}$ is a rotation tensor and $\mathbb{S}$ is a symmetric and positive definite matrix. We start from the relation
\begin{equation}
\label{eqmodb1}
D_t{\mathbb{F}}=a \nabla \f v \mathbb{F}+b (\nabla \f v)^T \mathbb{F},
\end{equation}
which becomes
\begin{equation}
\label{eqmodb2}
D_t{\mathbb{F}}=D_t({\mathbb{S}})\mathbb{Q} +\mathbb{S}D_t{\mathbb{Q}}=a \nabla \f v \mathbb{S} \mathbb{Q}+b (\nabla \f v)^T \mathbb{S} \mathbb{Q}.
\end{equation}
The latter relation involves $d^2$ equations, the tensor $\mathbb{Q}$ has $\frac{d(d-1)}{2}$ degrees of freedom and the tensor $\mathbb{S}$ has $\frac{d(d+1)}{2}$ degrees of freedom, so both $\mathbb{Q}$ and $\mathbb{S}$ can be determined by the previous relation. Multiplying \eqref{eqmodb2} by $\mathbb{Q}^T$ from the right, and introducing the angular velocity tensor $\mathbb{W}:=D_t({\mathbb{Q}})\mathbb{Q}^T$, we get that
\begin{equation}
\label{eqmodb3}
D_t({\mathbb{F}})\mathbb{Q}^T=D_t{\mathbb{S}}+\mathbb{S}\mathbb{W}.
\end{equation}
Also,
\begin{equation}
\label{eqmodb4}
\mathbb{Q}D_t\left({\mathbb{F}}^T\right)\mathbb{Q}=D_t{\mathbb{S}}-\mathbb{W} \mathbb{S}.
\end{equation}
Subtracting the latter two equations, we obtain that
\[
\mathbb{S}\mathbb{W}+\mathbb{W} \mathbb{S}=\mathbb{H}:=a[(\nabla \f v)\mathbb{S}-\mathbb{S}(\nabla \f v)^T]+b[(\nabla \f v)^T\mathbb{S}-\mathbb{S}(\nabla \f v)].
\]
The latter relation is a well known tensorial equation of the kind
\[
\mathbb{A} \mathbb{S}+ \mathbb{S} \mathbb{A}=\mathbb{H},
\]
with $\mathbb{S}\in \mathbb{R}_{sym}^{3\times 3}$, $\mathbb{A},\mathbb{H}\in 
\mathbb{R}_{skw}^{3\times 3}$ (in three space dimension). A solution of the latter equation is of the form (see e.g. \cite{sidoroff})
\begin{equation}
\label{eqmodb5}
\mathbb{A}=f(\mathbb{S})\mathbb{H}-g(\mathbb{S})(\mathbb{S}^2 \mathbb{H}+ \mathbb{H} \mathbb{S}^2),
\end{equation}
where $f(\mathbb{S}),g(\mathbb{S})$ are defined in terms of the linear invariants of $\mathbb{S}$ as
\[
f(\mathbb{S}):=\frac{I_{\mathbb{S}}^2-II_{\mathbb{S}}}{I_{\mathbb{S}}II_{\mathbb{S}}-III_{\mathbb{S}}}, \quad g(\mathbb{S}):=\frac{1}{I_{\mathbb{S}}II_{\mathbb{S}}-III_{\mathbb{S}}}.
\]
Note that \eqref{eqmodb5} is well defined if $\mathbb{S}$ is positive definite. Hence, we have that
\[
D_t{\mathbb{S}}+\mathbb{S}\mathbb{W}-a \nabla \f v \mathbb{S} -b (\nabla \f v)^T \mathbb{S} =\boldsymbol{0},
\]
with
\[
\mathbb{W}(\mathbb{S},\nabla \f v)=f(\mathbb{S})\mathbb{H}(\mathbb{S},\nabla \f v)-g(\mathbb{S})(\mathbb{S}^2 \mathbb{H}(\mathbb{S},\nabla \f v)+ \mathbb{H}(\mathbb{S},\nabla \f v) \mathbb{S}^2).
\]
We note that the latter model is equivalent to the model for $\mathbb{S}$ derived in \cite{balci}, where $\mathbb{S}$ is interpreted as the symmetric square root of the conformation tensor $\mathbb B$. In the present case we have derived more coincise formulas for the expression of $\mathbb{W}$, with a clearer mechanical and geometrical interpretation, then the ones reported in \cite[Appendix A]{balci}.

\subsection*{B: Relative energy inequality for the first example}

It has been shown in the proof of Theorem~\ref{thm:visco2} that the assumptions of Proposition~\ref{prop:rel} are fulfilled such that the relative energy inequality~\eqref{relenin} is fulfilled for the system~\eqref{sysvisco2}. The same holds for system~\eqref{sysvisco}. 
With the help of Corollary~\ref{cor:weakstrong}, we inferred the weak-strong uniqueness results from Theorem~\ref{thm:visco2} and Theorem~\ref{thm:visco}. 
Nevertheless, the calculation of the relative energy inequality~\eqref{relenin} remains a nonstandard task, at least for non purely quadratic energy and dissipation. 
We exemplify the calculations here for the convenience of the reader. 

By the definition of the relative energy in Proposition~\ref{prop:rel}, we observe for the system~\eqref{sysvisco2} 
\begin{align*}
    \mathcal{R}(\f v, \bB,E|\tv,\tbB ) &= E - \E(\tv,\tbB) - \left\langle \partial \E(\tv,\tbB), \begin{pmatrix}
        \f v- \tv\\\bB-\tbB 
    \end{pmatrix}\right \rangle \\
    &= E - \E(\f v , \bB) + \int_\Omega \frac{1}{2} | \f v - \tv|^2 + \frac\beta2\left[ | \bB - \MI|^2 - |\tbB - \MI|^2 -2 (\tbB-\MI)(\bB-\tbB)\right] \de x
\\&\quad + (1-\beta) \int_\Omega \left[ \tr(\bB-\MI-\ln(\bB))-\tr(\tbB-\MI-\ln(\tbB))-(\MI-\tbB^{-1})(\bB-\tbB) \right]\de x  
\\
&= E - \E(\f v , \bB) + \int_\Omega \frac{1}{2} | \f v - \tv|^2 \de x
\\
&\quad + \int_\Omega  \frac\beta2 | \bB-\tbB|^2 + (1-\beta) \tr(\tbB^{-1}\bB-\MI-\ln(\tbB^{-1}\bB))\de x
\,.
\end{align*}
Note that the relative energy is nonnegative, which can be observed for the last term by the fact that $\tbB^{-1}\bB$ is positive definite. 
In order to calculate the relative form $\mathcal{W}$, we insert the definition of the dissipation potential~\eqref{visco:diss2} and the system operator~\eqref{Asys2}, we find
\begin{align*}
    \mathcal{W}(\f v ,& \bB|\tv,\tbB) \coloneqq
    \\
   & \qquad\int_\Omega \mu | \nabla \f v - \nabla \tv|^2 + \tr\left( \tbB^{-1} \left[\tbB \bB^{-1} - 2\MI + (\tbB \bB^{-1})^{-1} \right] \right) 
    \\
    &\quad +(\beta+\delta (1-\beta)) \int_\Omega  \tr\left((\bB-\tbB)^2\right)
    \de x 
       + \delta \beta\int_\Omega  \tr \left(\left[ \bB + 2 \tbB - 2 \MI\right] ( \bB -\tbB)^2\right) \de x 
    \\
    &\quad +\int_\Omega \left((\f v - \tv) \otimes (\f v - \tv)+\beta (\bB - \tbB)^2 \right) : \nabla \tv \de x 
    \\
    &\quad + \int_\Omega  (\bB-\tbB)\otimes (\f v- \tv) \dreidots \nabla  \left[(1-\beta) (\MI-\tbB^{-1})+ \beta (\tbB-\MI) \right ]\de x 
    \\
    &\quad + \int_\Omega \left[(\nabla(\f v-\tv))_{\skw}+ \alpha ( \nabla (\f v-\tv))_{\skw}  \right](\bB-\tbB)  : \left[(1-\beta) (\MI-\tbB^{-1})+ \beta (\tbB-\MI) \right ]\de x 
    \\
    & \quad +\delta \int_\Omega (\bB-\tbB)^2 : \left[(1-\beta) (\MI-\tbB^{-1})+ \beta (\tbB-\MI) \right ]\de x 
    \\
&\quad + \mathcal{K}(\tv,\f B(\tbB)) \int_\Omega \frac{1}{2} | \f v - \tv|^2 + \frac\beta2 | \bB-\tbB|^2 + (1-\beta) \tr(\tbB^{-1}\bB-\MI-\ln(\tbB^{-1}\bB))\de x\,,
\end{align*}
where $\f B$ is defined in~\eqref{Bdef} and $\mathcal{K}$ in~\eqref{KB}. 
From the convexity assumption~\eqref{ass:convex}, which was shown to hold in the Proof of Theorem~\ref{thm:visco2}, we infer that $\mathcal{W}(\f v , \bB|\tv,\tbB) \geq 0$
for all solutions $(\f v, \bB)$ in the sense of Definition~\ref{def:visco2} and functions $(\tv,\tbB)$ fulfilling~\eqref{regStrongM}. In order to calculate $\mathcal{W}$, we inserted the simplifications
\begin{align*}
   \tr\left( \bB (\MI-\bB^{-1})^2 - \tbB ( \MI - \tbB)^2 - \left[ (\MI-\tbB^{-1})^2 + 2 ( \MI - \tbB^{-1})\tbB ^{-1}\right ](\bB-\tbB)\right)\\ = \tr \left( \tbB^{-1}\left[\tbB \bB^{-1} - 2\MI + (\tbB \bB^{-1})^{-1} \right]\right)
   \intertext{and}
   \tr\left(\bB(\bB-\MI)^2 - \tbB(\tbB-\MI)^2 -  \left[ (\tbB-\MI)^2+2 \tbB (\tbB-\MI) \right] ( \bB -\tbB)\right) \\= \tr \left(\left[ \bB + 2 \tbB - 2 \MI\right] ( \bB -\tbB)^2\right)\,.
\end{align*}
The second derivative of the energy was already calculated in~\eqref{secondE}. Inserting everything into the relative energy inequality~\eqref{relenin} and estimating $\mathcal{W}$ from below by zero, we end up with 
\begin{align*}
   \Big[ E &- \E(\f v , \bB + \int_\Omega \frac{1}{2} | \f v - \tv|^2 
+ \frac\beta2 | \bB-\tbB|^2 + (1-\beta) \tr(\tbB^{-1}\bB-\MI-\ln(\tbB^{-1}\bB))\de x\Big] \Bigg|_s^t \qquad\\
&+ \int_s^t
    \int_\Omega \left[ \t \tv + ( \tv \cdot \nabla ) \tv  - \mu \Delta \tv  - \alpha  \di \left ( (1-\beta) (\tbB-\MI)  +\beta ( \tbB^2 - \tbB) \right ) - \f f\right] (\f v- \tv) \de x \de \tau 
    \\
 & + \int_s^t  \int_\Omega   \left[ \t \tbB  -[ (\nabla \tv)_{\skw}  - \alpha    ( \nabla \tv )_{\sym} \tbB]_{\sym}\right] : \left[\beta(\bB-\tbB)+(1-\beta) \tbB^{-1}(\bB-\tbB)\tbB^{-1}\right] \de x \de \tau \\
& + \int_s^t  \int_\Omega \left[ ( \tv \cdot \nabla ) \tbB + \tbB - \MI + \delta ( \tbB^2 - \tbB)\right] :\left[\beta(\bB-\tbB)+(1-\beta) \tbB^{-1}(\bB-\tbB)\tbB^{-1}\right]  \de x 
\de \tau \\
\leq{}& \int_s^t\mathcal{K}(\tv,\f B(\tbB))  \left[E(t) - \E(\f v (t), \bB(t)) + \int_\Omega \frac{1}{2} | \f v - \tv|^2 \de x \right ]
\de \tau \qquad\qquad
\qquad\qquad
\\
&+ \int_s^t\mathcal{K}(\tv,\f B(\tbB))  \left[\int_\Omega  \frac\beta2 | \bB-\tbB|^2 + (1-\beta) \tr(\tbB^{-1}\bB-\MI-\ln(\tbB^{-1}\bB))\de x \right ]
\de \tau\,.
\end{align*}
This shows that the nonquadratic energy also requires a nonlinear testing of the equations in order to infer the proper weak-strong uniqueness result. As in Corollary~\ref{cor:weakstrong}, we may infer weak strong uniqueness from the above inequality, but it can also be interpreted as a continuous dependence result as long as a strong solution exists. The relative energy inequality is also used in the literature to analyse multiple scales via singular limits~\cite{singular}, estimate errors of model simplifications~\cite{fischer}, optimal control and numerical approximation~\cite{approx}.

\small

  \end{document}